\documentclass[final]{siamltex}

\usepackage{amsmath}
\usepackage{amsbsy}
\usepackage{amssymb}
\usepackage{textcomp}
\usepackage{graphicx}
\usepackage{amsfonts}
\usepackage{longtable}
\usepackage{multirow}
\usepackage{lscape}
\usepackage{url}
\usepackage{subfig}
\usepackage{longtable}
\usepackage{lscape}
\usepackage{hhline}
\usepackage{colortbl,array} 
\usepackage{multirow,bigdelim}
\usepackage{booktabs}
\usepackage[ruled]{algorithm2e}
\usepackage{epstopdf}

\newtheorem{thm}{Theorem}[section]

\newtheorem{exam}[thm]{Example.\nopagebreak}

\def\eeq{\end{equation}} 
\def\lbeq#1{\begin{equation} \label{#1}} 

\def\eps{\varepsilon} 
\def\D{\displaystyle}				
\def\ol{\overline}
\def\ul{\underline}
\def\fct#1{\mathop{\rm #1}}	
\def\argmax{\fct{argmax}}

\def\argmin{\fct{argmin}}

\def\half{\frac{1}{2}}
\def\wh{\widehat}

\def\bary{\begin{array}}
\def\eary{\end{array}}

\def\st{\fct{s.t.~}}
\def\x{\mathbf{x}}
\def\ul{\underline}
\def\ol{\overline}
\def\<{\langle}
\def\>{\rangle}
\def\Rz{\mathbb{R}}

\def\gzit#1{(\ref{#1})}
\def\iif{~~~\mbox{if }}
\def\half{\frac{1}{2}} 

\title{An optimal subgradient algorithm for large-scale bound-constrained convex optimization
       }

\author{
Masoud Ahookhosh\thanks{Faculty of Mathematics, University of Vienna,
Oskar-Morgenstern-Platz 1, 1090 Vienna, Austria. ({\tt masoud.ahookhosh@univie.ac.at})}
\and Arnold Neumaier\thanks{Faculty of Mathematics, University of Vienna,
Oskar-Morgenstern-Platz 1, 1090 Vienna, Austria. ({\tt Arnold.Neumaier@univie.ac.at})}
        }
\begin{document}
\maketitle
\slugger{mms}{xxxx}{xx}{x}{x--x}

\begin{abstract}
This paper shows that the OSGA algorithm -- which uses first-order information to solve convex optimization problems with optimal complexity -- can be used to efficiently solve arbitrary bound-constrained convex optimization problems. This is done by constructing an explicit method as well as an inexact scheme for solving the bound-constrained rational subproblem required by OSGA. This leads to an efficient implementation of OSGA on large-scale problems in applications arising signal and image 
processing, machine learning and statistics. Numerical experiments demonstrate the promising performance of OSGA on such problems. ions to show the efficiency of the proposed scheme. A software package implementing OSGA for bound-constrained convex problems is available.
\end{abstract}



\pagestyle{myheadings}
\thispagestyle{plain}
\markboth{\sc M. Ahookhosh and A. Neumaier}{\sc OSGA for bound-constrained convex optimization}

\section{Introduction}
Let $\mathcal{V}$ be a finite-dimensional real vector space and $\mathcal{V}^*$ its dual space. In this paper we consider the bound-constrained convex minimization problem
\begin{equation}\label{e.gfun}
\begin{array}{ll}
\min & f(\mathcal{A}x)\\
\st  & x \in \x,
\end{array}
\end{equation}
where $f: \x \rightarrow \Rz$ is a -- smooth or nonsmooth -- convex function, $\mathcal{A}: \mathbb{R}^n \rightarrow \mathbb{R}^m$ is a linear operator, and $\x = [\ul{x}, \ol{x}]$ is an axiparallel box in $\mathcal{V}$ in which $\ul{x}$ and $\ol{x}$ are the vectors of lower and upper bounds on the components of $x$, respectively. (Lower bounds are allowed 
to take the value $-\infty$, and upper bounds the value $+\infty$.)

Throughout the paper, $\<g,x\>$ denotes the value of $g \in \mathcal{V}^*$ at $x \in \mathcal{V}$. A subgradient of the objective function $f$ at $x$ is a vector 
$g(x)\in \mathcal{V}^*$ satisfying
\[
f(z) \geq f(x) + \< g(x),z - x \>,
\]
for all $z \in \mathcal{V}$. It is assumed that the set of optimal solutions of (\ref{e.gfun}) is nonempty and the first-order information about the objective function (i.e., 
for any $x\in\x$, the function value $f(x)$ and some subgradient $g(x)$ at $x$) are available by a first-order black-box oracle.\\
\textbf{Motivation \& history.} 
Bound-constrained optimization in general is an important problem appearing in many fields of science and engineering, where the parameters describing physical quantities are constrained to be in a given range. Furthermore, it plays a prominent role in the development 
of general constrained optimization methods since many methods reduce the solution of the general problem to the solution of a sequence of bound-constrained problems. 

There are  lots of bound-constrained algorithms and solvers for smooth and nonsmooth optimization; here, we mention only those related to our study. {\sc Lin \& Mor\'e} in \cite{LinM} and {\sc Kim} et al. in \cite{KimSD} proposed Newton and quasi-Newton methods for solving bound-constrained optimization. In 1995, {\sc Byrd} et al. \cite{ByrLNZ} proposed a limited memory algorithm called LBFGS-B for general smooth nonlinear bound-constrained optimization. {\sc Branch} et al. in \cite{BraCL} proposed a trust-region method to solve this problem. {\sc Neumaier \& Azmi} \cite{NeuA} solved this problem by a limited memory algorithm. The smooth bound-constrained optimization problem was also solved by {\sc Birgin} et al. in \cite{BirMR} and {\sc Hager \& Zhang} in \cite{HagZ1,HagZ2} using nonmonotone spectral projected gradient methods, active set strategy and affine scaling scheme, respectively. Some limited memory bundle methods for solving bound-constrained nonsmooth problems were
proposed by {\sc Karmitsa \& M\"akel\"a} \cite{KarM1, KarM2}. 

In recent years convex optimization has received much attention 
because it arises in many applications and is suitable for solving problems
involving high-dimensional data. The particular case of bound-constrained 
convex optimization involving a smooth or nonsmooth objective function 
also appears in a variety of applications, of which we mention the following:
\\\\
{\bf Bound-constrained linear inverse problems.} 
Given $A, W \in \mathbb{R}^{m \times n}$, $b \in \mathbb{R}^m$ and 
$\lambda \in \mathbb{R}$, for $m \geq n$, the bound-constrained least-squares problem is given by
\begin{equation}\label{e.bcls}
\begin{array}{ll}
\min  &~~ \D f(x) := \frac{1}{2} \| Ax - b \|_2^2 + \lambda \varphi(x)\\
\st  &~~ x \in \x,
\end{array}
\end{equation}
and the bound-constrained $l_1$ problem is given by
\begin{equation}\label{e.bcl1}
\begin{array}{ll}
\min  &~~ \D f(x) := \| Ax - b \|_1 + \lambda \varphi(x)\\
\st  &~~ x \in \x,
\end{array}
\end{equation}
where $\x=[\underline{x}, \overline{x}]$ is a box and $\varphi$ is a 
smooth or nonsmooth regularizer, often a weighted power of a norm; 
see Section \ref{s.num} for examples.

The problems (\ref{e.bcls}) and (\ref{e.bcl1}) are commonly arising in the context of control and inverse problems, especially for some imaging problems like denoising, deblurring and inpainting.  {\sc Morini} et al. \cite{MorPC} formulated the bound-constrained least-squares 
problem (\ref{e.bcls}) as a nonlinear system of equations and proposed an iterative method based on a reduced Newton's method. Recently, {\sc Zhang \& Morini} 
\cite{ZhaM} used alternating direction methods to solve these problems. More recently, {\sc Chan} et al. in \cite{ChaTY1}, {\sc Bo?} et al. in \cite{BotCH}, 
and {\sc Bo? \& Hendrich} in \cite{BotH} proposed alternating direction methods, primal-dual splitting methods, and a Douglas-Rachford primal-dual 
method, respectively, to solve both (\ref{e.bcls}) and (\ref{e.bcl1}) for some applications.\\
{\bf Content.} 
In this paper we show that the optimal subgradient algorithm OSGA proposed by {\sc Neumaier} in  
\cite{NeuO} can be used for solving bound-constrained problems of the 
form (\ref{e.gfun}). In order to run
OSGA, one needs to solve a rational auxiliary subproblem. We here investigate efficient schemes for solving this subproblem in the presence of bounds on its
variables. To this end, we show that the solution of the subproblem has a
one-dimensional piecewise linear representation and that it may be computed by solving a sequence 
of one-dimensional piecewise rational optimization problems.
We also give an iterative scheme that 
can solve OSGA's subproblem approximately by solving a 
one-dimensional nonlinear equation. We give numerical 
results demonstrating 
the performance of OSGA on some problems from applications. More specifically, 
in the next section, we investigate properties of the solution of the subproblem 
(\ref{e.sub}) that lead to two algorithms for solving it efficiently. 
In Section 3, we report numerical results with artificial and
real data, where to solve large-scale imaging problems we use
the inexact scheme to find the solution of OSGA's subproblem.  
Finally, Section 4 gives some conclusions.

\section{Overview of OSGA} 
OSGA (see Algorithm 1) is an optimal subgradient algorithm proposed by {\sc Neumaier} 
in \cite{NeuO} that constructs -- for a problem \gzit{e.gfun} with 
arbitrary nonempty, closed and convex domain, not necessarily a box 
-- a sequence of iterates whose function values converge with the optimal
complexity to the minimum of \gzit{e.gfun}. Furthermore, OSGA requires
no information regarding global parameters like Lipschitz constants of function
values and gradients.

Apart from first-order information at the iterates, OSGA requires an efficient routine for finding a maximizer $\wh x=U(\gamma,h)$ and the optimal objective value $E(\gamma,h)$ of an auxiliary problem of the form
\lbeq{e.sub}
\begin{array}{ll}
\sup & E_{\gamma,h}(x)\\
\st  &  x \in \x,
\end{array}
\eeq
where it is known that the supremum is positive. The function $E_{\gamma,h}: \x \rightarrow \Rz$ is defined by
\begin{equation}\label{e.ex}
 E_{\gamma,h}(x):= -\frac{\gamma+\<h,x\>}{Q(x)}
\end{equation}
with $\gamma\in\Rz$, $h\in \mathcal{V}^*$, and 
$Q: \x \to \Rz $ is a prox-function, i.e., a strongly 
convex function satisfying $ \inf_{x\in \x} Q(x) >0$ and
\lbeq{e.strc}
Q(z)\ge Q(x)+\<g_Q(x),z-x\>+\frac{\sigma}{2}\|z-x\|^2, 
\eeq
where the convexity parameter is $\sigma = 1$. In particular, 
with $Q_0>0$, if $x \in \mathbb{R}^n$, we may take
\lbeq{e.prox}
Q(x):= Q_0 + \half \|x-x^0\|_2^2 \iif x \in \mathbb{R}^n,
\eeq
where $\|x\|_2 = \sqrt{\sum_i x_i^2}$ is the Euclidean norm, and 
\lbeq{e.prox1}
Q(x):= Q_0 + \half \|x-x^0\|_F^2 \iif x \in \mathbb{R}^{m \times n},
\eeq
where $\|x\|_F = \sqrt{\sum_{i,k} x_{ik}^2}$ is the Frobenius norm.

In \cite{Aho, NeuO}, the unconstrained version (\ref{e.sub}) with the 
prox-function (\ref{e.prox}) or (\ref{e.prox1}) is solved in a closed form.
Numerical experiments and comparisons with some state-of-the-art solvers \cite{Aho, AhoN3,AhoN4} show the promising behaviour of OSGA for solving practical unconstrained problems. A version of OSGA that can solves structured nonsmooth problems with the complexity $O(\varepsilon^{-1/2})$ discussed in \cite{AhoN2}.
In \cite{AhoN1}, it is shown that the OSGA subproblem can be solved for many simple domains.
In this paper we show that for the above choices of $Q(x)$ and an
arbitrary box $\x$, the subproblem \gzit{e.sub} can be solved 
efficiently. It follows that OSGA is applicable to solve 
bound-constrained convex problems as well. 

It is shown in {\sc Neumaier} \cite{NeuO} that OSGA satisfies the 
optimal complexity bounds $O\left(\varepsilon^{-2} \right)$ for 
Lipschitz continuous nonsmooth convex functions and 
$O\left(\varepsilon^{-1/2} \right)$ for smooth convex functions 
with Lipschitz continuous gradients; cf.  
{\sc Nemirovsky \& Yudin} \cite{NemY} and {\sc Nesterov} \cite{NesB, NesS}. Since the underlying problem 
(\ref{e.gfun}) is a special case of the problem considered in 
\cite{NeuO}, the complexity of OSGA remains valid for (\ref{e.gfun}).

As discussed in \cite{NeuO}, OSGA uses the following scheme for 
updating the given parameters $\alpha$, $h$, $\gamma$, $\eta$,
and $u$:\\

\begin{algorithm}[h] \label{a.par}
\DontPrintSemicolon 
\KwIn{ $\delta$,~ $\alpha_{\max}\in{]0,1[}$,~ $0<\kappa'\le\kappa$, $\alpha$, $\eta$, $\bar{h}$, $\bar{\gamma}$, $\bar{\eta}$, $\bar{u}$;}
\KwOut{$\alpha$,~ $h$,~ $\gamma$,~ $\eta$,~ $u$;}
\Begin{
    $R \gets \left(\eta-\ol \eta)/(\delta\alpha \eta \right)$;\;
    \eIf{$R<1$} {
        $h \gets \ol h$;\;
    }{
        $\ol\alpha \leftarrow \min(\alpha e^{\kappa' (R-1)},\alpha_{\max})$;
    }
    $\alpha \gets \ol \alpha$;\;
    \If{$\ol \eta<\eta$}{
        $h \gets \ol h$;~ $\gamma \gets \ol \gamma$;~ $\eta \gets \ol \eta$;~ $u \gets \ol u$;\;
    }
}
\caption{ {\bf PUS} (parameters updating scheme)}
\label{algo:max}
\end{algorithm}

\vspace{3mm}
\begin{algorithm}[H] \label{a.osga}
\DontPrintSemicolon 
\KwIn{global parameters: $\delta, \alpha_{\max}\in{]0,1[}$,~ $0<\kappa'\le\kappa$;~ local parameters: $x_0$,~$\mu \geq 0$,~ $f_{\mathrm{target}}$;}
\KwOut{$x_b$,~ $f_{x_b}$;}
\Begin{
    choose an initial best point $x_b$;\;
    compute $f_{x_b}$ and $g_{x_b}$;\;
    \eIf{$f_{x_b} \leq f_{\mathrm{target}}$} {
        stop;\;
    }{
        $h = g_{x_b}-\mu g_Q(x_b)$;~ $\gamma = f_{x_b}-\mu Q(x_b)-\<h,x_b\>$;\;
        $\gamma_b = \gamma - f_{x_b}$;~ $u = U(\gamma_b,h)$;~ $\eta = 
        E(\gamma_b,h)-\mu$;\;
    }
    $\alpha \gets \alpha_{\max}$;\\
    \While {stopping criteria do not hold}{
        $x = x_b+\alpha(u-x_b)$; compute $f_x$ and $g_x$;\;
        $g = g_x-\mu g_Q(x)$;~ $\ol h = h+\alpha(g-h)$;\;
        $\ol\gamma = \gamma+\alpha(f_x - \mu Q(x)-\<g,x\>-\gamma)$;\;
        $x_b' = \argmin_{z \in \{x_b,x\}} f(z, v_z)$;~ $f_{x_b'} = 
        \min \{f_{x_b}, f_x\}$;\;
        $\gamma_b' = \ol\gamma-f_{x_b'}$;~ $u' = U(\gamma_b',\ol h)$;\; 
        $x' = x_b+\alpha(u'-x_b)$; compute $f_{x'}$;\;
        choose $\ol x_b$ in such a way that $f_{\ol x_b}\le \min\{f_{x_b'},f_{x'} \}$;\;
        $\ol\gamma_b = \ol\gamma - f_{\ol x_b}$;~ $\ol u = U(\ol \gamma_b,\ol h)$;~
        $\ol \eta = E(\ol\gamma_b,\ol h)-\mu$; ~$x_b = \ol x_b$; ~ $f_{x_b} = f_{\ol x_b}$;\;
        \eIf {$f_{x_b} \leq f_{\mathrm{target}}$}{
            stop;\;
        }{
            update the parameters $\alpha$, $h$, $\gamma$, $\eta$ and $u$;
        }
    }
}
\caption{ {\bf OSGA} (optimal subgradient algorithm)}
\end{algorithm}

\section{ Solving OSGA's rational subproblem (\ref{e.sub})}
In this section we investigate the solution of the 
bound-constrained subproblem (\ref{e.sub}) and give two iterative schemes, 
where the first one solves (\ref{e.sub}) exactly and the 
second one solves it approximately.

\subsection{Explicit solution of OSGA's rational subproblem (\ref{e.sub})}
In this subsection we describe an explicit solution of the 
bound-constrained subproblem (\ref{e.sub}). 

Without loss of generality, we here consider $\mathcal{V} = \mathbb{R}^n$. It is not hard to adapt the results to $\mathcal{V} = \mathbb{R}^{n \times n}$ and other spaces.
The method is related to one used in several earlier papers. In 1980, {\sc Helgason} et al.  \cite{HelKL} characterized the solution of singly constrained quadratic problem with bound constraints. Later, {\sc Paradolas \& Kovoor}  \cite{ParK} developed an $O(n)$ algorithm for this problem using binary search to solve the associated
Kuhn-Tucker system. This problem was also solved by 
{\sc Dai \& Fletcher} \cite{DaiF} using a projected gradient method. 
{\sc Zhang} et al.  \cite{ZhaSV} solved the linear support vector machine problem by a cutting plane method employing a similar technique.

In the papers mentioned, the key is showing that the problem 
can be reduced to a piecewise linear problem in a single dimension. To apply this idea to the present problem, we prove that (\ref{e.sub}) is equivalent to an one-dimensional minimization problem and then develop a procedure to calculate its minimizer. We write
\begin{equation}\label{e.sol}
x(\lambda):=\sup \{ \ul{x},\inf\{x^0-\lambda h,\ol{x}\} \}
\end{equation} 
for the projection of $x^0-\lambda h$ to the box $\x$.

\begin{proposition}\label{p.1}
For $h \neq 0$, the maximum of the subproblem (\ref{e.sub}) is attained at
$\wh x:=x(\lambda)$, where $\lambda>0$ or $\lambda = + \infty$ is the inverse of the value of the maximum.
\end{proposition}

\begin{proof}
The function $E_{\gamma,h}: \mathcal{V}  \rightarrow \Rz $ 
defined by (\ref{e.ex}) is continuously differentiable and 
from Proposition 2.1 in \cite{NeuO} we have $e:= E_{\gamma,h} > 0$. 
By differentiating both side of the 
equation $E_{\gamma,h}(x) Q(x) = -\gamma - \< h,x \>$, we obtain
\begin{equation*}
\frac{\partial E_{\gamma,h}}{\partial x} Q(x) =  - e (x - x^0) - h.
\end{equation*}
At the maximizer $\wh x$, we have 
$e Q(\wh x) = -\gamma - \< h,\wh x \>$. Now the first-order optimality 
conditions imply that for $i = 1, 2, \ldots, n$,
\begin{equation}
-e (\wh x_i - x_i^0) - h_i \left \{
\begin{array}{lll}\label{e.kkt}
\leq 0 &~~ \mathrm{if}~ \wh x_i = \ul{x}_i,\\
\geq 0 &~~ \mathrm{if}~ \wh x_i = \ol{x}_i,\\
= 0    &~~ \mathrm{if}~ \ul{x}_i < \wh x_i < \ol{x}_i.
\end{array}
\right.
\end{equation}
Since $e > 0$, we may define  $\lambda := e^{-1}$ and find that,
for $i = 1, 2, \ldots, n$,
\begin{equation}\label{e.zs}
\wh x_i = \left \{
\begin{array}{lll}
\ul{x}_i           &~~ \mathrm{if}~ \ul{x}_i \geq x_i^0 -\lambda h_i,\\
\ol{x}_i           &~~ \mathrm{if}~ \ol{x}_i \leq x_i^0 -\lambda h_i,\\
x_i^0 -\lambda h_i &~~ \mathrm{if}~ 
                        \ul{x}_i \leq x_i^0 -\lambda h_i \leq \ol{x}_i.
\end{array}
\right.
\end{equation}
This implies that $\wh x = x(\lambda)$. 
\end{proof}

Proposition \ref{p.1} gives the key feature of the solution of the subproblem
(\ref{e.sub}) implying that it can be considered in the form of (\ref{e.sol}) with only one variable $\lambda$. In the remainder of this section, we focus on deriving the optimal $\lambda$. 

\begin{exam}\label{e.nonneg}
Let first consider a very special case that $\x$ is nonnegative orthant, i.e., $\underline{x}_i= 0$ and $\overline{x}_i = +\infty$, for $i=1,\cdots,n$. Nonnegativity constraint is important in many applications, see \cite{BarV,ElfHN,EssLX,KauN1,KauN2}. In this case we consider the quadratic function
\lbeq{e.prox11}
Q(z):= \half \|z\|_2^2 + Q_0,
\eeq
where $Q_0 > 0$. In \cite{Aho}, it is proved that (\ref{e.prox11}) is a prox-function. By using this prox-function and (\ref{e.sol}), we obtain 
\[
x(\lambda) = \sup \{ \ul{x},\inf\{-\lambda h,\ol{x}\} = - \lambda (h)_-,
\]
where $(z)_- = \min \{0,z\}$. By Proposition 2.2 of \cite{NeuO}, we have 
\[
\frac{1}{\lambda} \left( \frac{1}{2} \|x(\lambda)\|_2^2 + Q_0 \right) + \gamma + \langle h,x(\lambda) \rangle = \left( \frac{1}{2} \|(h)_-\|_2^2 + \langle h,x(\lambda) \rangle  \right) \lambda^2 + \gamma \lambda + Q_0 = 0.
\]
By substituting $\lambda = 1/e$ into this equation, we get
\[
\beta_1 e^2 + \beta_2 \lambda + \beta_3 = 0,
\]
where $\beta_1 = Q_0$, $\beta_2 = \gamma$, and $\beta_3 =\frac{1}{2} \|(h)_-\|_2^2 - \langle h,(h)_-$.
Since we search for the maximum $e$, the solution is the bigger root of this equation, i.e., 
\[
e =  \frac{-\beta_2 + \sqrt{\beta_2^2-4 \beta_1 \beta_3}}{2\beta_1}.
\]
This shows that for the nonnegativity constraint the subproblem (\ref{e.sub}) can be solved in a closed form, however, for a general bound-constrained problem, we need a much more sophisticated scheme to solve (\ref{e.sub}).  
\end{exam}

To derive the optimal $\lambda \geq 0$ in Proposition \ref{p.1}, we first determine its permissible range provided
by the three conditions considered in (\ref{e.zs}) leading to the 
interval 
\begin{equation}\label{e.inti}
\lambda \in [\ul{\lambda}_i,\ol{\lambda}_i],
\end{equation}
for each component of $x$. In particular, if $h_i = 0$, since $x^0$
is a feasible point, $\wh x_i = x_i^0 -\lambda h_i = x_i^0$ satisfies the third 
condition in (\ref{e.zs}). Thus there is no upper bound for $\lambda$, leading to
\begin{equation}\label{e.lamr0}
\D \ul{\lambda}_i = 0, ~~\ol{\lambda}_i = + \infty 
~~~ \mathrm{if}~ \wh x_i = x_i^0,~ h_i=0.
\end{equation}
If $h_i \neq 0$, we consider the three cases 
(i) $\ul{x}_i \geq x_i^0 -\lambda h_i$,
(ii) $\ol{x}_i \leq x_i^0 -\lambda h_i$, and 
(iii) $\ul{x}_i \leq x_i^0 -\lambda h_i \leq \ol{x}_i$ of (\ref{e.zs}).
In Case (i), if $h_i < 0$, division by $h_i$ implies that 
$\lambda \leq - (\ul{x}_i - x_i^0)/h_i \leq 0$, which is not in the acceptable range
for $\lambda$. In this case, if $h_i > 0$ then
$\lambda \geq - (\ul{x}_i - x_i^0)/h_i$ leading to
\begin{equation}\label{e.lamr1}
\D \ul{\lambda}_i = -\frac{\ul{x}_i - x_i^0}{h_i},~\ol{\lambda}_i = + \infty 
~~~ \mathrm{if}~ \wh x_i = \ul{x}_i,~h_i > 0.
\end{equation}
In Case (ii), if $h_i < 0$ then $\lambda \geq - (\ol{x}_i - x_i^0)/h_i$ implying
\begin{equation}\label{e.lamr2}
\D \ul{\lambda}_i= - \frac{\ol{x}_i - x_i^0}{h_i},~\ol{\lambda}_i= + \infty 
~~~ \mathrm{if}~ \wh x_i = \ol{x}_i,~ h_i < 0.
\end{equation} 
In Case (ii), if $h_i > 0$ then $\lambda \leq - (\ol{x}_i - x_i^0)/h_i \leq 0$, which
is not in the acceptable range of $\lambda$. In Case (iii), if $h_i < 0$, division by 
$h_i$ implies 
\[
- \frac{\ul{x}_i - x_i^0}{h_i} \leq \lambda \leq - \frac{\ol{x}_i - x_i^0}{h_i}.
\] 
The lower bound satisfies $- (\ul{x}_i - x_i^0)/h_i \leq 0$, so it is not acceptable, leading to 
\begin{equation}\label{e.lamr3}
\D \ul{\lambda}_i= 0,~\ol{\lambda}_i =-\frac{\ol{x}_i - x_i^0}{h_i}
~~~ \mathrm{if}~ \wh x_i = x_i^0 -\lambda h_i \in [\ul{x}_i \ol{x}_i],~ h_i < 0.
\end{equation}
In Case (iii), if $h_i > 0$, then 
\[
- \frac{\ol{x}_i - x_i^0}{h_i} \leq \lambda \leq - \frac{\ul{x}_i - x_i^0}{h_i}.
\] 
However, the lower bound $- (\ol{x}_i - x_i^0)/h_i \leq 0$ is not acceptable, i.e.,
\begin{equation}\label{e.lamr4}
\D \ul{\lambda}_i= 0,~\ol{\lambda}_i =-\frac{\ul{x}_i - x_i^0}{h_i}
~~~ \mathrm{if}~ \wh x_i = x_i^0 -\lambda h_i \in [\ul{x}_i \ol{x}_i],~ h_i > 0.
\end{equation}
As a result, the following proposition is valid.

\begin{proposition} \label{p.interval}
If $x(\lambda)$ is solution of the problem (\ref{e.sub}), then 
\[
\lambda \in [\underline{\lambda}_i,\overline{\lambda}_i] ~~~  i = 1, \cdots, n,
\]
where $\underline{\lambda}_i$ and $\overline{\lambda}_i$ 
are computed by
\begin{equation}\label{e.lambda0}
\underline{\lambda}_j = \left \{
\begin{array}{ll}
\D -\frac{\ul{x}_i - x_i^0}{h_i}    & \mathrm{if}~ \wh x_i = \ul{x}_i,~ h_i > 0,\\
\D - \frac{\ol{x}_i - x_i^0}{h_i}    & \mathrm{if}~ \wh x_i = \ol{x}_i,~ h_i < 0,\\
\D 0 & \mathrm{if}~ \widetilde{x}_i \in [\ul{x}_i \ol{x}_i],~ h_i < 0,\\
\D 0 & \mathrm{if}~ \widetilde{x}_i \in [\ul{x}_i \ol{x}_i],~ h_i > 0,\\
\D 0                                & \mathrm{if}~ h_i=0,
\end{array}
\right. 
\overline{\lambda}_j = \left \{
\begin{array}{ll}
\D + \infty     & \mathrm{if}~ \wh x_i = \ul{x}_i,~ h_i > 0,\\
\D + \infty & \mathrm{if}~ \wh x_i = \ol{x}_i,~ h_i < 0,\\
\D -\frac{\ol{x}_i - x_i^0}{h_i}     & \mathrm{if}~ \widetilde{x}_i \in [\ul{x}_i \ol{x}_i],~ h_i < 0,\\
\D -\frac{\ul{x}_i - x_i^0}{h_i}     & \mathrm{if}~ \widetilde{x}_i \in [\ul{x}_i \ol{x}_i],~ h_i > 0,\\
\D + \infty                            & \mathrm{if}~ h_i=0,
\end{array}
\right.
\end{equation}
in which $\widetilde{x}_i = x_i^0 -\lambda h_i$ for $i = 1, \cdots, n$.
\end{proposition}

Proposition \ref{p.interval} implies that each element of $x$ satisfies in only one of the conditions 
(\ref{e.lamr0})--(\ref{e.lamr4}). Thus, for each $i = 1, \dots, n$ with $h_i \neq 0$, we have a 
single breakpoint 
\begin{equation}\label{e.bre}
\widetilde{\lambda}_i := \left \{
\begin{array}{ll}
\D- \frac{\ol{x}_i - x_i^0}{h_i} &~~ \mathrm{if}~ h_i <0,\\
\D- \frac{\ul{x}_i - x_i^0}{h_i} &~~ \mathrm{if}~ h_i >0,\\
\D +\infty                        &~~ \mathrm{if}~ h_i = 0.
\end{array}
\right.
\end{equation}

Sorting the $n$ bounds $\widetilde{\lambda}_i, ~ i = 1, \dots, n$, in increasing order, augmenting the resulting 
list by $0$ and $+\infty$, and deleting possible duplicate 
points, we obtain a list of $m + 1 \leq n + 2$ different breakpoints, 
denoted by
\begin{equation}\label{e.bpoi}
0 = \lambda_1 < \lambda_2< \ldots < \lambda_{m}< \lambda_{m +1} 
= + \infty.
\end{equation}
By construction, $x(\lambda)$ is linear in each interval $[\lambda_k, \lambda_{k+1}]$, for $k = 1, \cdots, m$. The next proposition gives an explicit representation for $x(\lambda)$.

\begin{proposition}\label{p.2}
The solution $x(\lambda)$ of the auxiliary problem (\ref{e.sub}) defined by (\ref{e.sol}) has the form
\begin{equation}\label{e.xlam}
x(\lambda) = p^k + \lambda q^k ~~~ 
\mathrm{for}~ \lambda \in [\lambda_k,\lambda_{k+1}]~~~ 
(k = 1,2, \ldots, m),
\end{equation}
where 
\begin{equation}\label{e.pkqk}
p_i^k = \left \{
\begin{array}{ll}
x_i^0    &~~ \mathrm{if}~ h_i = 0,\\
x_i^0    &~~ \mathrm{if}~ \lambda_{k+1} \leq \widetilde{\lambda}_i,\\
\ul{x}_i &~~ \mathrm{if}~ \lambda_k \geq \widetilde{\lambda}_i, ~ h_i > 0,\\
\ol{x}_i &~~ \mathrm{if}~ \lambda_k \geq \widetilde{\lambda}_i, ~ h_i < 0,
\end{array}
\right. ~~~~~~
q_i^k = \left \{
\begin{array}{ll}
0     &~~ \mathrm{if}~ h_i = 0,\\
- h_i &~~ \mathrm{if}~ \lambda_{k+1} \leq \widetilde{\lambda}_i,\\
0     &~~ \mathrm{if}~ \lambda_k \geq \widetilde{\lambda}_i, ~ h_i > 0,\\
0     &~~ \mathrm{if}~ \lambda_k \geq \widetilde{\lambda}_i, ~ h_i < 0.
\end{array}
\right.
\end{equation}
\end{proposition}

\begin{proof}
Since $\lambda > 0$, there exists $k \in \{1, \cdots, m\}$ such that $\lambda \in [\lambda_k,\lambda_{k+1}]$. Let $i \in \{1, \cdots, n\}$. If $h_i = 0$, (\ref{e.lamr0})
implies $\wh x_i = x_i^0$. If $h_i \neq 0$, the way of construction of $\lambda_i$ for $i = 1, \cdots, m$ implies that $\widetilde{\lambda}_i \not\in (\lambda_k,\lambda_{k+1})$, so two cases are distinguished: 
(i) $\lambda_{k+1} \leq \widetilde{\lambda}_i$; 
(ii) $\lambda_k \geq \widetilde{\lambda}_i$. 
In Case (i), Proposition \ref{p.interval} implies that $\widetilde{\lambda}_i = \overline{\lambda}_i$, while it is not possible $\widetilde{\lambda}_i \neq \underline{\lambda}_i$. Therefore, either (\ref{e.lamr3}) or (\ref{e.lamr4}) holds dependent on the sign of $h_i$, implying $x_i^0 -\lambda h_i \in [\underline{x}_i, \overline{x}_i]$, 
so that $p_i^k = x_i^0$ and $q_i^k = -h_i$. 
In Case (ii), Proposition \ref{p.interval} implies that $\widetilde{\lambda}_i = \underline{\lambda}_i$, while it is not possible $\widetilde{\lambda}_i \neq \overline{\lambda}_i$. Therefore, either (\ref{e.lamr1}) or (\ref{e.lamr2}) holds. If $h_i < 0$, then (\ref{e.lamr2}) holds, i.e., $p_i^k = \ol{x}_i$ and $q_i^k = 0$. 
Otherwise, (\ref{e.lamr1}) holds, implying $p_i^k = \ul{x}_i$ and $q_i^k = 0$.
This proves the claim.   
\end{proof}

Proposition \ref{p.2} exhibits the solution $x(\lambda)$ of the 
auxiliary problem (\ref{e.sub}) as a piecewise linear function of $\lambda$.
In the next result, we show that solving the problem (\ref{e.sub}) is 
equivalent to maximizing a one-dimensional piecewise rational function.

\begin{proposition}\label{p.3}
The maximal value of the subproblem (\ref{e.sub}) is the maximum of 
the piecewise rational function $e(\lambda)$ defined by 
\begin{equation}\label{e.lam}
e(\lambda) 
:= \frac{a_k + b_k \lambda}{c_k + d_k \lambda + s_k \lambda^2} 
~~~\mbox{if}~~ \lambda \in [\lambda_k,\lambda_{k+1}]~~~
(k= 1, 2, \ldots, m),
\end{equation}  
where
\[
a_k := -\gamma - \< h, p^k \>,~~~ 
b_k := -\< h, q^k \>, 
\]
\[
c_k := Q_0 + \frac{1}{2} \| p^k - x^0\|^2, ~~~
d_k := \< p^k - x^0, q^k \>, ~~~
s_k := \frac{1}{2} \|q^k\|^2.
\] 
Moreover, $c_k > 0$, $c_k > 0$ and $4s_k c_k > d_k^2$.
\end{proposition}

\begin{proof}
By Proposition \ref{p.1} and \ref{p.2}, the global minimizer of 
(\ref{e.sub}) has the form (\ref{e.xlam}).
For $k = 1, 2, \ldots, m$ and $\lambda \in [\lambda_k,\lambda_{k+1}]$,
we substitute (\ref{e.xlam}) into the function (\ref{e.ex}), and obtain
\begin{equation}\label{e.xlami}\begin{split}
E_{\gamma,h}(x(\lambda)) 
&= -\frac{\gamma + \< h, x^k(\lambda) \>}{Q(x^k(\lambda))} 
= -\frac{\gamma + \< h, p^k + q^k \lambda \>}{Q(p^k + q^k \lambda)}\\
& = -\frac{\gamma + \< h, p^k \> + \< h, q^k \> \lambda}
{Q_0 + \frac{1}{2} \| p^k - x^0\|^2 + \< p^k - x^0, q^k \> \lambda 
     + \frac{1}{2} \|q^k\|^2 \lambda^2} = e(\lambda),
\end{split}\end{equation}
as defined in the proposition. 

Since $Q_0 > 0$, we have $c_k > 0$ and the denominator 
of (\ref{e.lam}) is bounded away from zero, implying $4s_k c_k > d_k^2$. It is 
enough to verify $s_k >0$. The definition of $q_k$ in (\ref{e.pkqk}) implies
that $h_i \neq 0$ for $i \in I = \{i ~:~ \lambda_{k+1} \leq \widetilde{\lambda}_i\}$
leading to $q^k \neq 0$ implying $s_k > 0$. 
\end{proof}

The next result leads to a systematic way to maximize the  
one-dimensional rational problem (\ref{e.lam}). 

\begin{proposition}\label{p.4}
Let $a$, $b$, $c$, $d$, and $s$ be real constants with $c > 0$, $s > 0$, and $4sc > d^2$. Then  
\begin{equation}\label{e.hlam}
\phi(\lambda) := \frac{a + b \lambda}{c + d \lambda + s \lambda^2}
\end{equation}  
defines a function $\phi:\Rz \rightarrow \Rz$ that has at least one
stationary point. Moreover, the global maximizer of $\phi$ is determined
by the following cases:\\\\
(i) If $b \neq 0$, then $a^2 - b (ad - bc)/s > 0$ and the global maximum
\begin{equation}\label{e.phil2}
\phi(\wh \lambda) = \frac{b}{2s \wh \lambda + d}
\end{equation}
is attained at
\begin{equation}\label{e.lam2}
\wh \lambda = \frac{-a + \sqrt{a^2 - b(ad - bc)/s}}{b}.
\end{equation}
(ii) If $b = 0$ and $a > 0$, the global maximum is
\begin{equation}\label{e.phil3}
\phi(\wh \lambda) = \frac{4as}{4cs - d^2},
\end{equation}
attained at
\begin{equation}\label{e.lam3}
\wh \lambda = -\frac{d}{2s}.
\end{equation}
(iii) If $b = 0$ and $a \leq 0$, the maximum is $\phi(\wh \lambda) = 0$, attained at 
$\wh \lambda = + \infty$ for $a<0$ and at all $\lambda \in \Rz$ for $a=0$.
\end{proposition}

\begin{proof}
The denominator of (\ref{e.hlam}) is positive for all $\lambda \in \mathbb{R}^n$ iff the 
stated condition on the coefficients hold. By the differentiation of $\phi$ 
and using the first-order optimality condition, we obtain
\begin{equation*}
\phi'(\lambda) 
= \frac{b(c + d\lambda + s\lambda^2) - (a + b\lambda) (d + 2s\lambda)}
       {(c + d \lambda + s \lambda^2)^2} 
= - \frac{bs \lambda^2 + 2as \lambda + ad -bc}
         {(c + d \lambda + s \lambda^2)^2}.
\end{equation*}
For solving $\phi'(\lambda) = 0$, we consider possible solutions of the quadratic
equation $bs \lambda^2 + 2as \lambda + ad -bc = 0$. Using the assumption 
$4sc > d^2$, we obtain
\[
\begin{split}
(2as)^2 - 4 bs (ad - bc) &= (2as)^2 - (4 abds - 4b^2cs)\\
&= (2as)^2 - 4 abds - b^2 (d^2 -4cs - d^2)\\
&= (2as)^2 - 4 abds + (bd)^2 - b^2 (d^2 -4cs)\\
& \geq (2as - bd)^2 - b^2 (d^2 -4cs) \geq 0,
\end{split}
\]
leading to
\[
a^2 - \frac{b(ad-bc)}{s} \geq 0,
\]
implying that $\phi'(\lambda) = 0$ has at least one solution. 

(i) If $b \neq 0$, then
\[
a^2 - \frac{b(ad-bc)}{s} = a^2 -\frac{bd}{s}a - \frac{b^2c}{s} = \left( a - \frac{bd}{2s} \right)^2 + \frac{b^2}{4s^2}(4sc-d^2) > 0,
\]
implying there exist two solutions. Solving $\phi'(\lambda) = 0$, 
the stationary points of the function are found to be 
\begin{equation}\label{e.lamb10}
\lambda = \frac{-a \pm \sqrt{a^2 - b(ad - bc)/s}}{b}.
\end{equation}
Therefore, $a + b \lambda = \pm w$ with
\[
w := \sqrt{a^2 - b(ad - bc)/s} > 0,
\] 
and we have
\begin{equation}\label{e.phil}
\phi(\lambda) = \frac{\pm w}{c + d \lambda + s \lambda^2}.
\end{equation} 
Since the denominator of this fraction is positive and $w \geq 0$, 
the positive sign in equation (\ref{e.lamb10}) gives the maximizer, 
implying that (\ref{e.lam2}) is satisfied. Finally, substituting this 
maximizer into (\ref{e.phil}) gives
\begin{equation*}\begin{split}
\phi(\wh \lambda) 
&= \frac{w}{c + d \wh \lambda + s \wh \lambda^2} 
= \frac{b^2 w}{b^2c + bd (w - a) + s (w - a)^2}\\
&=  \frac{b^2 w}{a^2 s - b(ad - bc) + s w^2 + (bd - 2 as)w} 
=  \frac{b^2 w}{2 s w^2 + (bd - 2 as)w}\\
&=  \frac{b^2 w}{w (2s (w - a) + bd)} 
=  \frac{b}{2s \wh \lambda + d},
\end{split}\end{equation*}
so that (\ref{e.phil2}) holds. 

(ii) If $b = 0$, we obtain
\[
\phi'(\lambda) =  \frac{- a (d + 2 s \lambda)}{(c + d \lambda + s \lambda^2)^2}.
\]
Hence the condition $\phi'(\lambda) = 0$ implies that $a = 0$ or $d + 2 s \lambda = 0$.
The latter case implies 
\[
\wh \lambda =- \frac{d}{2s}, ~~~ \phi(\wh \lambda) = 
\frac{4as}{4cs - d^2},
\]
whence $\wh \lambda$ is a stationary point of $\phi$. If $a > 0$, its maximizer is $\wh \lambda =- \frac{d}{2s}$ and (\ref{e.phil3})
is satisfied. 

(iii) If $b=0$ and $a < 0$, then
\[
\lim_{\lambda \rightarrow -\infty} \phi(\lambda) = \lim_{\lambda \rightarrow +\infty} \phi(\lambda) = 0
\]
implies $\phi(\wh \lambda) = 0$ at $\wh \lambda = \pm \infty$. In case $a = 0$, 
$\phi(\lambda) = 0$ for all $\lambda \in \Rz$. 
\end{proof}

We summarize the results of Propositions \ref{p.1}--\ref{p.4} into the 
following algorithm for computing the global optimizer $x_b$ and the 
optimum $e_b$ of (\ref{e.sub}).\\\\

\begin{algorithm}[H] \label{a.subprob}
\DontPrintSemicolon 
\KwIn{$Q_0, x^0, h, \ul{x}, \ol{x}$;}
\KwOut{$ u_b = U(\gamma, h), e_b = e(x_b)$;}
\Begin{
    \For{ $i = 1, 2, \ldots, n$ }{
        find $\widetilde{\lambda}_i$ by (\ref{e.bre}) using $\ul{x}$ and $\ol{x}$;
    }
    determine the breakpoints $\lambda_k,~ k=1, \dots, m+1$, by (\ref{e.bpoi});\;
    $e_b = 0$;\;
    \For{ $k = 1, 2, \ldots, m $ }{
        compute $p^k$ and $q^k$ using (\ref{e.pkqk});\;
        construct $e(\lambda)$  using (\ref{e.lam}) for 
            $[\lambda_k,\lambda_{k+1}]$;\;
        find the maximizer $\wh \lambda$ of $e(\lambda)$ 
            using Proposition \ref{p.4};\;
        \eIf {$\wh \lambda \in [\lambda_k, \lambda_{k+1}]$}{
            compute $e^k = e(\wh \lambda)$ using Proposition \ref{p.4};\;
            $\wh \lambda^k = \wh \lambda$;\;
        }{
            $e^k = \max\{e(\lambda_k), e(\lambda_{k+1})\}$;\;
            $\wh \lambda^k = \argmax_{i \in \{k, k+1\}} \{e(\lambda_i)\}$;
        }
        $E(k) = e^k$,~ $LAM(k) = \wh \lambda^k$;\;     
    }
    $j = \argmax\{E(i) \mid i = 1, \cdots, m \} $;\;
    $e_b = E(j),~ \wh \lambda = LAM(j),~ u_b = x^0 - \wh \lambda h$;\;
        
}
\caption{ {\bf BCSS} (bound-constrained subproblem solver)}
\end{algorithm}

\vspace{7mm}
\subsection{Inexact solution of OSGA's rational subproblem (\ref{e.sub})}
In this subsection we give a scheme to compute an inexact solution for the the subproblem (\ref{e.sub}). 

We here use the quadratic prox-function (\ref{e.prox11}). In view of Proposition \ref{p.1} and Theorem 3.1 in \cite{AhoN1}, the solution of the subproblem (\ref{e.sub}) is given by $x(\lambda)$ defined in (\ref{e.sol}), where $\lambda$ can be computed by solving the one-dimensional nonlinear equation 
\[
\varphi(\lambda) = 0,
\]
in which 
\lbeq{e.none}
\varphi(\lambda) := \frac{1}{\lambda} \left(\frac{1}{2}\| x(\lambda)\|_2^2 + Q_0 \right) + \gamma + \langle h, x(\lambda) \rangle.
\eeq
The solution of OSGA's subproblem can be found by Algorithm 3 (OSS) in \cite{AhoN1}. In \cite{AhoN1}, it is shown that in many convex domains the nonlinear equation (\ref{e.none}) can be solved explicitly, however, for the bound-constrained problems it can be only solved approximately. 

As discussed in \cite{AhoN1}, the one-dimensional nonlinear equation can be solved by some zero-finder schemes such as the bisection method and the secant bisection scheme described in Chapter 5 of \cite{NeuB}.
One can also use the MATLAB's $\mathtt{fzero}$ function combining the bisection scheme, the inverse quadratic interpolation, and the secant method. In the next section we will use this inexact solution of OSGA's rational subproblem (\ref{e.sub}) for solving large-scale imaging problems.

\vspace{1mm}
\section{Applications and numerical experiments} \label{s.num}
In this section we report some numerical results to show the performance of OSGA compared with some state-of-the-art algorithms on both artificial and real data. 

A software package implementing OSGA for solving unconstrained and bound-constrained convex optimization problems is publicly available at
\begin{center}
\url{http://homepage.univie.ac.at/masoud.ahookhosh/}.
\end{center}
The package is written in MATLAB, where the parameters 
\begin{equation*}
\delta = 0.9;~~ \alpha_{max} = 0.7;~~ \kappa = \kappa' = 0.5;~~ \Psi_{\mathrm{target}} = - \infty 
\end{equation*}
are used. We also use the prox-function (\ref{e.prox}) with 
$Q_0 = \frac{1}{2} \|x_0\|_2 + \epsilon$, where $\epsilon$ is the machine precision. Some examples for each class of problems are available to show how the user can implement it. The interface to each subprogram in the package is fully documented in the corresponding file. Moreover, the OSGA user's manual \cite{AhoUM} describes the design of the package and how the user can solve his/her own problems.
 
The algorithms considered in the comparison use the default parameter values
reported in the corresponding papers or packages. All implementations were executed on a 
Toshiba Satellite Pro L750-176 laptop with Intel Core i7-2670QM Processor and 8 GB RAM.

\subsection{Experiment with artificial data} 
In this subsection we deal with solving the problem (\ref{e.gfun}) with the objective functions
\lbeq{e.prob}
\begin{array}{ll}
\vspace{1mm} f(x)= \frac{1}{2} \|Ax-b\|_2^2 + \frac{1}{2} \|x\|_2^2 &~~~ (L22L22R),\\
\vspace{1mm} f(x)= \frac{1}{2} \|Ax-b\|_2^2 + \|x\|_1               &~~~ (L22L1R),\\
\vspace{1mm} f(x)= \|Ax-b\|_1 + \frac{1}{2} \|x\|_2^2               &~~~ (L1L22R),\\
             f(x)= \|Ax-b\|_1 + \|x\|_1                             &~~~ (L1L1R).\\
\end{array}
\eeq
The problem is generated by
\[
\mathtt{[A,z,x] = i\_laplace(n),~~~ b = z + 0.1*rand,}
\]
where $n$ is the problem dimension and $\mathtt{i\_laplace.m}$ is a code generating an ill-posed test problem using the inverse Laplace transformation from the Regularization Tools MATLAB package, which is available at
\begin{center}
\url{http://www.imm.dtu.dk/~pcha/Regutools/}.
\end{center}
The upper and lower bounds on variables are generated by 
\[
\mathtt{\underline{x}=0.05 * ones(n)},~~~
\mathtt{\overline{x}=0.95 * ones(n)},
\]
respectively. Since among the problems (\ref{e.prob}) only 
L22L22R is differentiable, we need some nonsmooth algorithms to be compared with OSGA. In our experiment 
we consider two versions of OSGA, i.e., a version uses BCSS for solving the subproblem (\ref{e.sub}) (OSGA-1) and a version uses the inexact solution described in Subsection 3.2 for solving the subproblem (\ref{e.sub}) (OSGA-2), compared with PSGA-1 (a projected subgradient algorithm with nonsummable diminishing step-size), and PSGA-2 (a projected subgradient algorithm with nonsummable diminishing steplength), see \cite{BoyXM}.

We solve all of the above-mentioned problems with the dimensions $n = 2000$ and $n = 5000$. The results for L22L22R and L22L1R are illustrated in Table 1 and Figure 1, and the results for L1L22R and L1L1R are summarized in Table 2 and Figure 2. More precisely, Figures 1 and 2 show the relative error of function vales versus iterations
\lbeq{e.delta}
\delta_k := (f_k - \wh f)/(f_0 - \wh f),
\eeq
where $\wh f$ denotes the minimum and $f_0$ shows the function value on an initial point $x_0$. In our experiments, PSGA-1 and PSGA-2 exploit the step-sizes $\alpha := 5/\sqrt{k}\|g_k\|$ and $\alpha := 0.1/\sqrt{k}$, respectively, in which $k$ is the iteration counter and $g_k$ is a subgradient of $f$ at $x_k$. The algorithms are stopped after 100 iterations.

\begin{table}[htbp]
\caption{Result summary for L22L22R and L22L1R}
\begin{center}\footnotesize
\renewcommand{\arraystretch}{1.3}
\begin{tabular}{|l|l|l|l|l|l|l|}\hline
\multicolumn{1}{|l|}{} & \multicolumn{1}{l|}{{\bf problem}}& \multicolumn{1}{l|}{{\bf dimension}}
&\multicolumn{1}{l|}{{\bf PSGA-1}} & \multicolumn{1}{l|}{{\bf PSGA-2}} &\multicolumn{1}{l|}{{\bf OSGA-1}} & \multicolumn{1}{l|}{{\bf OSGA-2}} \\ 
\hline
{\bf $f_b$} & L22L22R & $n= 2000$ & 81.8427 & 77.1302 & 77.1285 & 77.1285\\
{\bf Time } &  &                    & 0.74 & 0.75 & 4.15 & 3.08\\
\hline
{\bf $f_b$} & L22L22R & $n = 5000$ & 4.7561e+2 & 4.2646e+2 & 4.2640e+2 & 4.2645e+2\\
{\bf Time } &  &                    & 3.67 & 3.58 & 14.09 & 7.57\\
\hline
{\bf $f_b$} & L22L1R & $n= 2000$ & 1.8922e+2 & 1.8827e+2& 1.8682e+2 & 1.2367e+2\\
{\bf Time } &  &                    & 0.67 & 0.61 & 3.91 & 1.21\\
\hline
{\bf $f_b$} & L22L1R & $n= 5000$ & 7.0679e+2 & 6.8084e+2& 6.7887e+2 & 6.8064e+2\\
{\bf Time } &  &                    & 3.72 & 3.42 & 14.20 & 7.61\\
\hline
\end{tabular}
\end{center}
\end{table}

\begin{table}[htbp]
\caption{Result summary for L1L22R and L1L1R}
\begin{center}\footnotesize
\renewcommand{\arraystretch}{1.3}
\begin{tabular}{|l|l|l|l|l|l|l|}\hline
\multicolumn{1}{|l|}{} & \multicolumn{1}{l|}{{\bf problem}}& \multicolumn{1}{l|}{{\bf dimension}}
&\multicolumn{1}{l|}{{\bf PSGA-1}} & \multicolumn{1}{l|}{{\bf PSGA-2}} &\multicolumn{1}{l|}{{\bf OSGA-1}} & \multicolumn{1}{l|}{{\bf OSGA-2}} \\ 
\hline
{\bf $f_b$} & L1L22R & $n= 2000$ & 1.8981e+2 & 1.9420e+2 & 1.8671e+2 & 1.8676e+2\\
{\bf Time } &  &                    & 0.69 & 0.75 & 4.04 & 2.73\\
\hline
{\bf $f_b$} & L1L22R & $n = 5000$ & 1.9256e+2 & 3.4612e+2 & 1.6971e+2 & 1.6995e+2\\
{\bf Time } &  &                    & 3.69 & 3.57 & 14.06 & 6.83\\
\hline
{\bf $f_b$} & L1L1R & $n= 2000$ & 2.6713e+2 & 2.6936e+2& 2.6536e+2 & 2.6703e+2\\
{\bf Time } &  &                    & 0.76 & 0.75 & 4.27 & 3.02\\
\hline
{\bf $f_b$} & L1L1R & $n= 5000$ & 6.9728e+2 & 7.4536e+2& 6.9411e+2 & 6.9687e+2\\
{\bf Time } &  &                    & 3.68 & 3.57 & 14.46 & 7.46\\
\hline
\end{tabular}
\end{center}
\end{table}

\begin{figure}[h] \label{f.l22}
\centering
\subfloat[][L22L22R, $n=2000$]{\includegraphics[width=6.1cm]{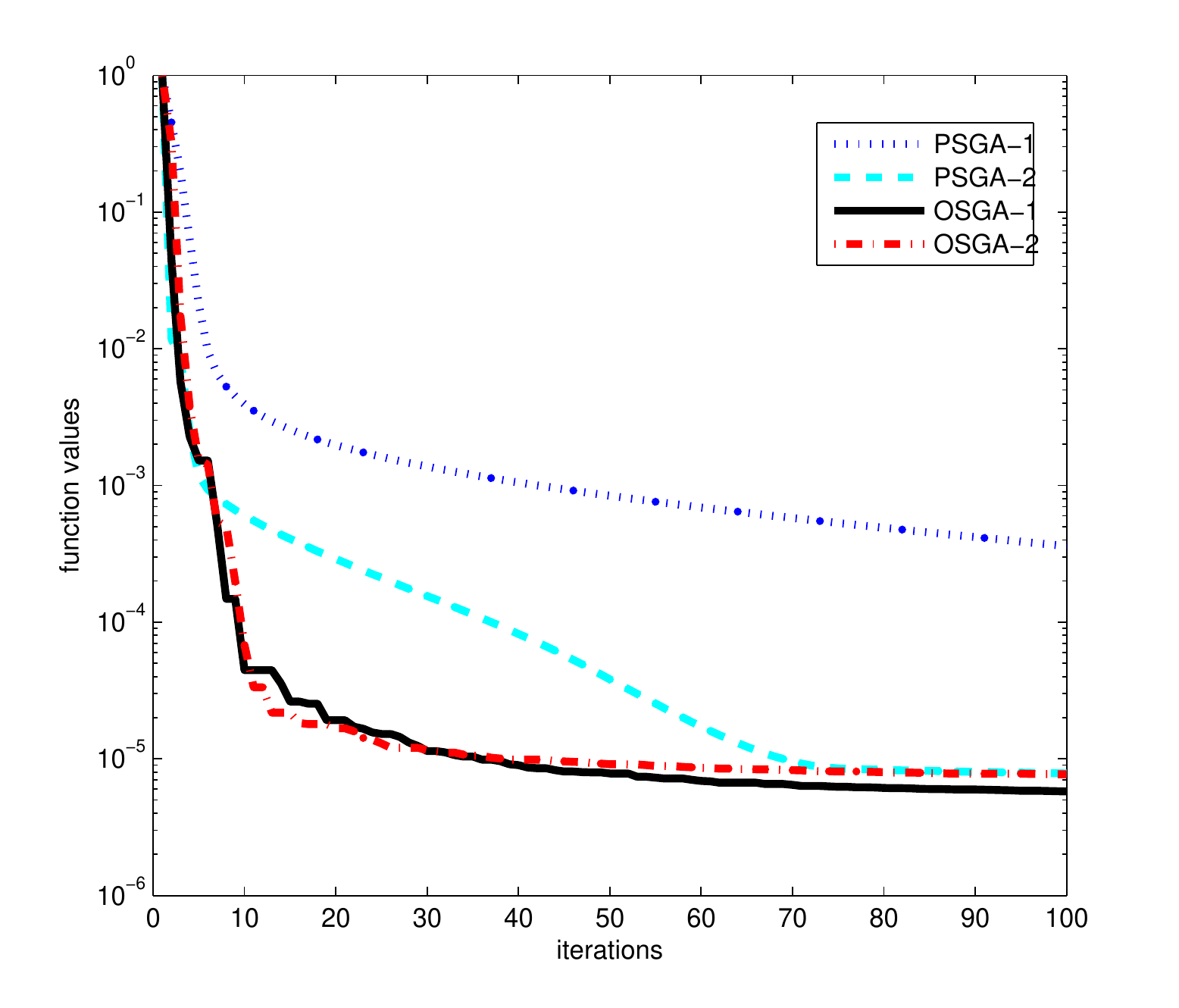}}%
\qquad
\subfloat[][L22L22R, $n=5000$]{\includegraphics[width=6.1cm]{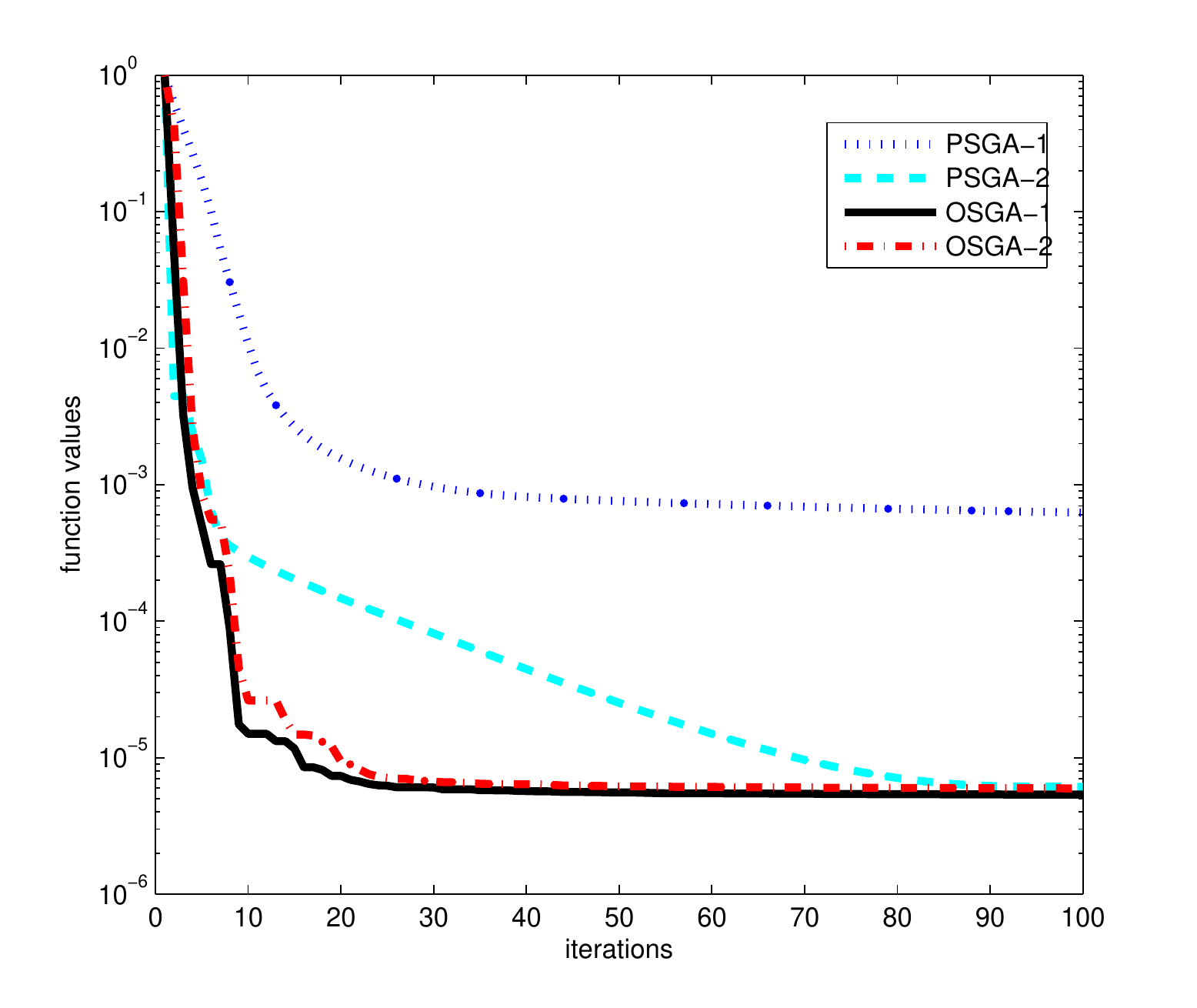}}
\qquad
\subfloat[][L22L1R, $n=2000$]{\includegraphics[width=6.1cm]{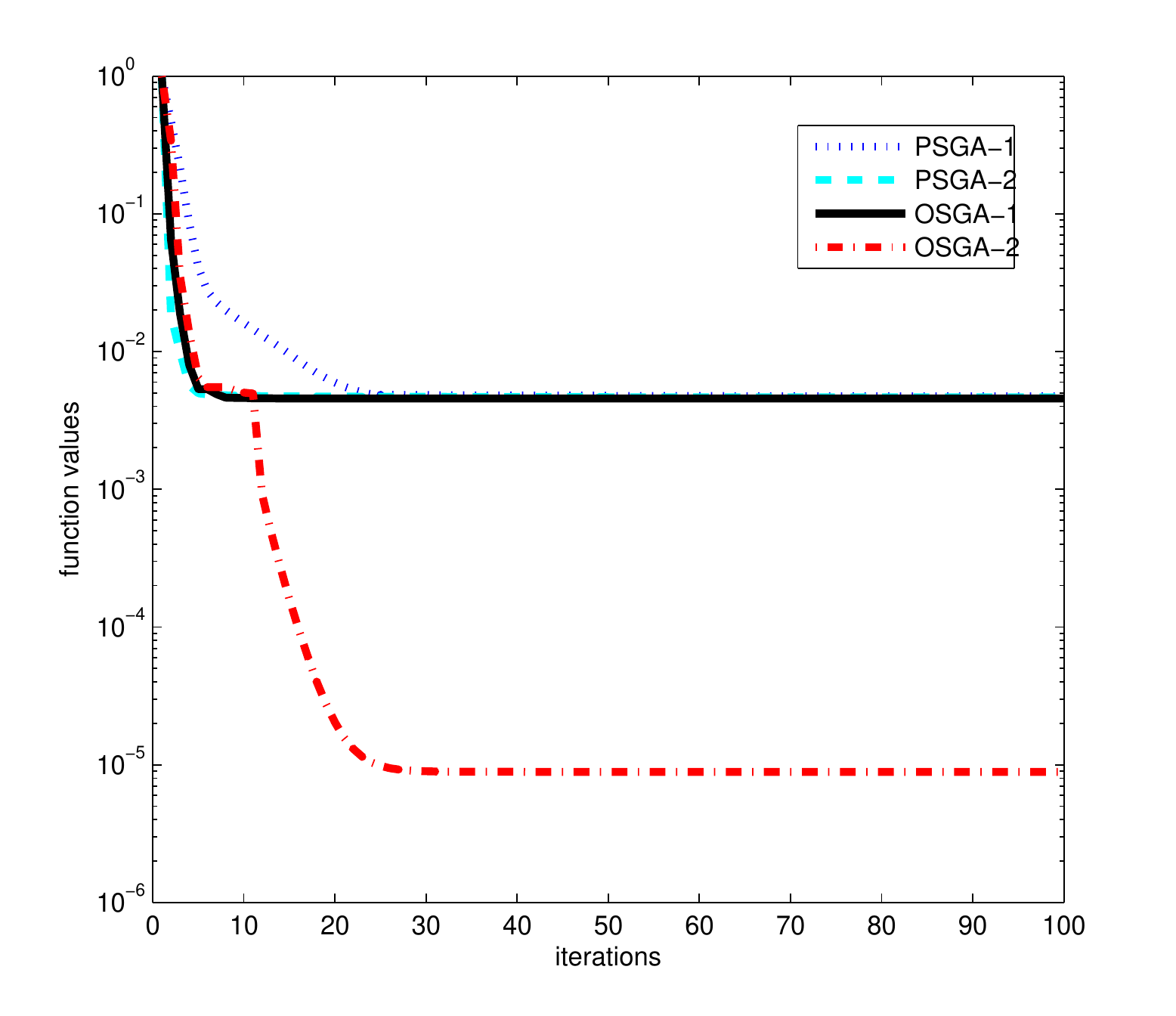}}%
\qquad
\subfloat[][L22L1R, $n=5000$]{\includegraphics[width=6.1cm]{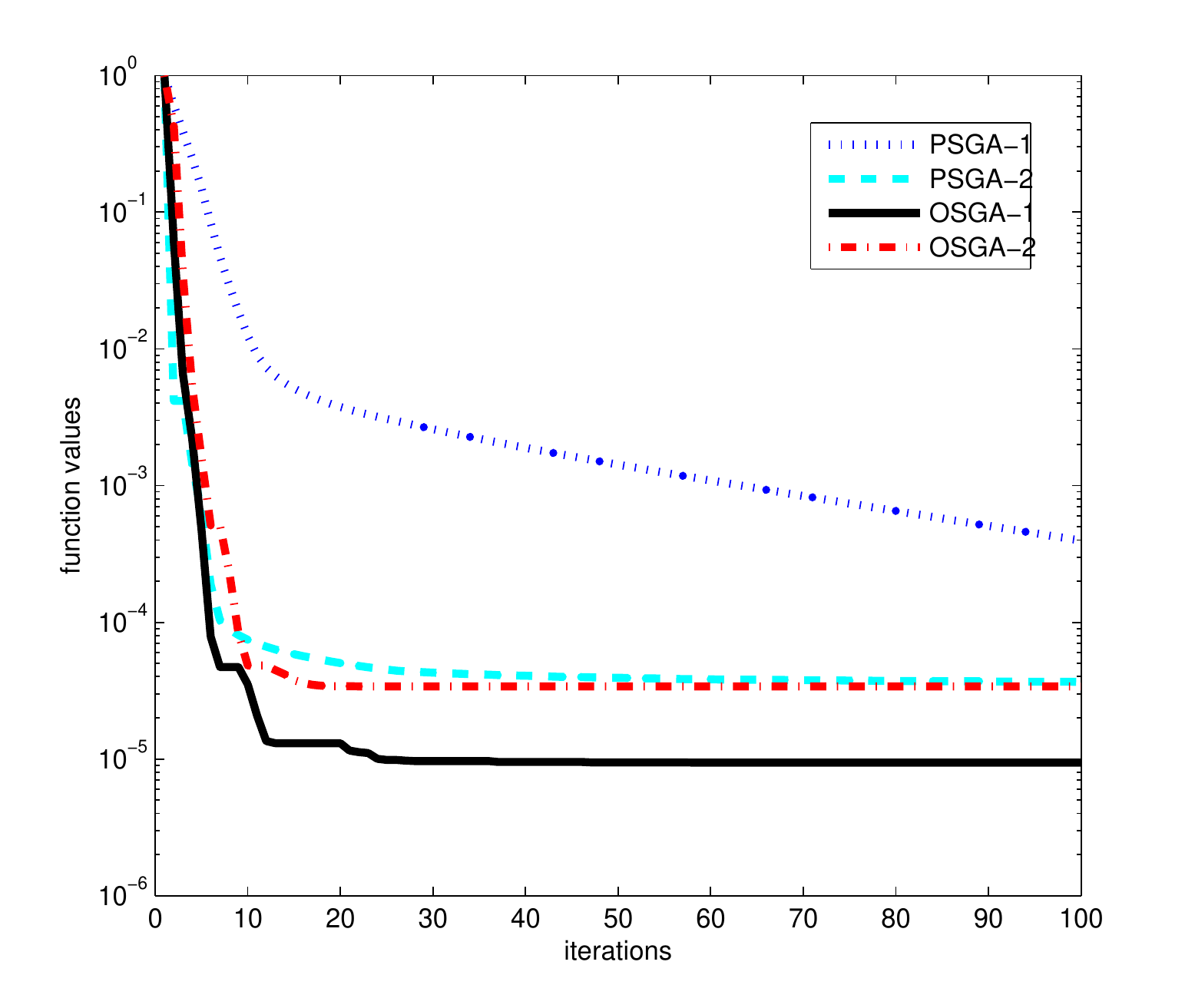}}

\caption{A comparison among PSGA-1, PSGA-2, OSGA-1, and OSGA-2 for solving L22L22R and L22L1R, where the algorithms were stopped after 100 iterations.}
\end{figure}

\begin{figure}[h] \label{f.l1}
\centering
\subfloat[][L1L22R, $n=2000$]{\includegraphics[width=6.1cm]{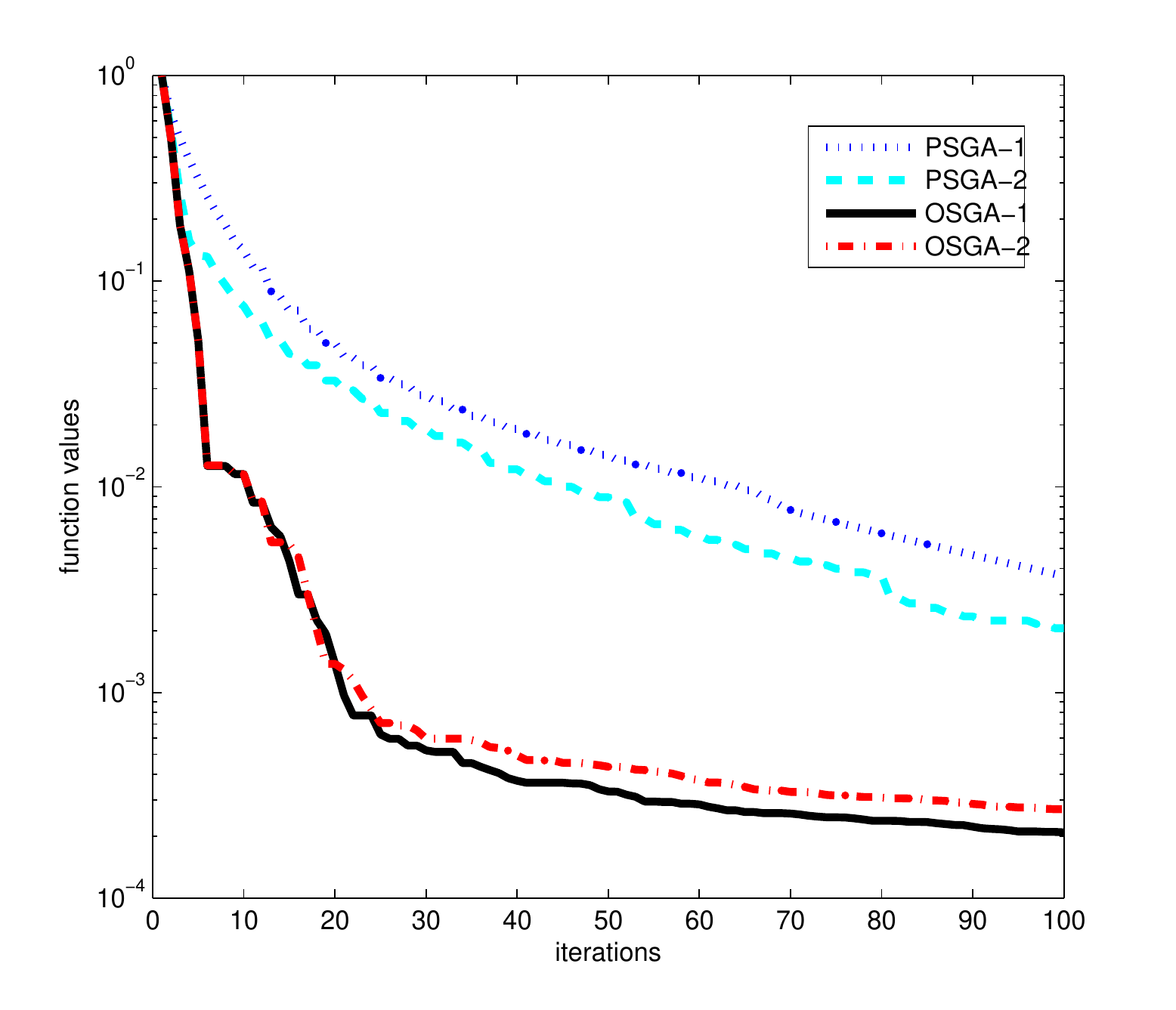}}%
\qquad
\subfloat[][L1L22R, $n=5000$]{\includegraphics[width=6.1cm]{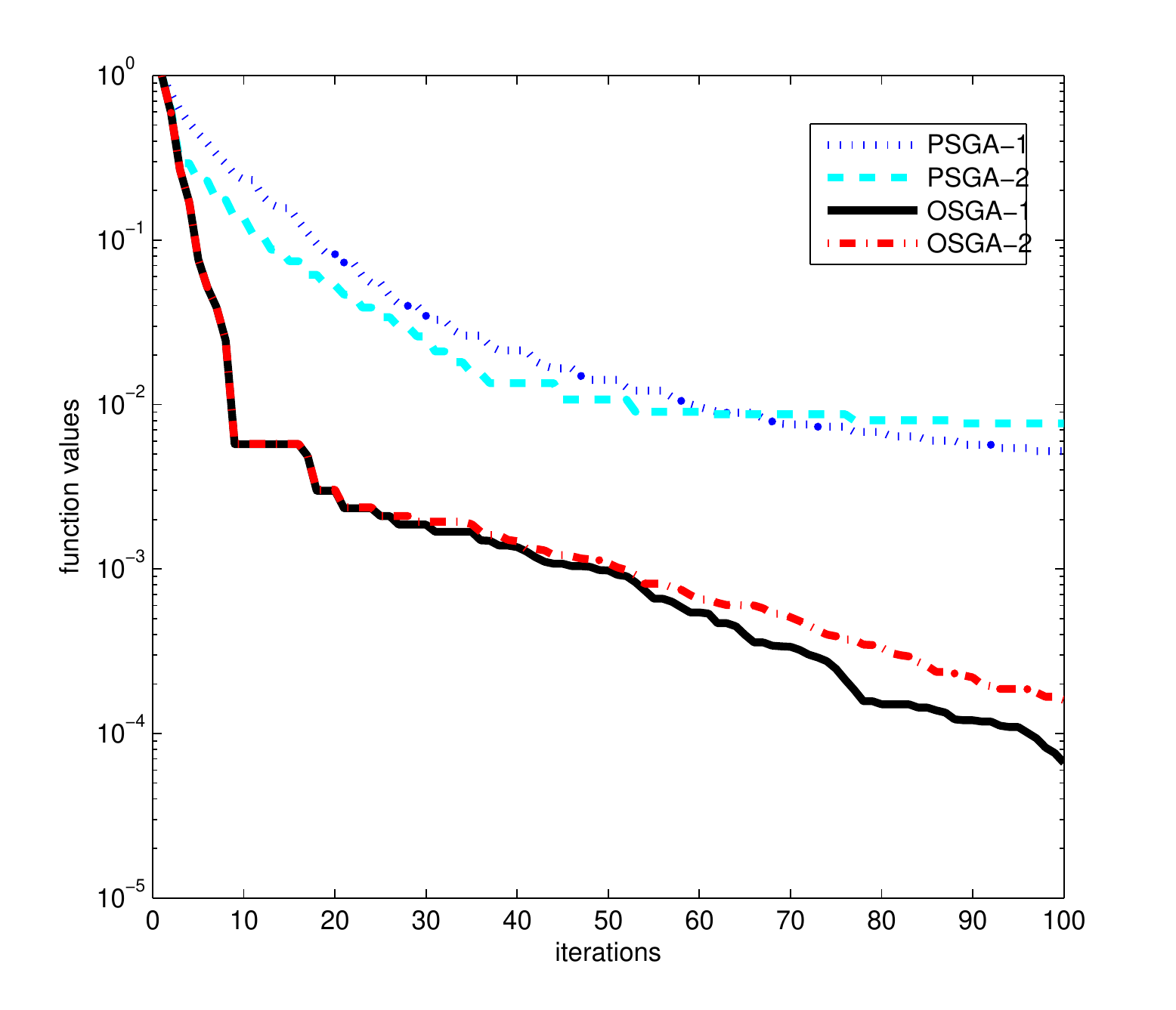}}
\qquad
\subfloat[][L1L1R, $n=2000$]{\includegraphics[width=6.1cm]{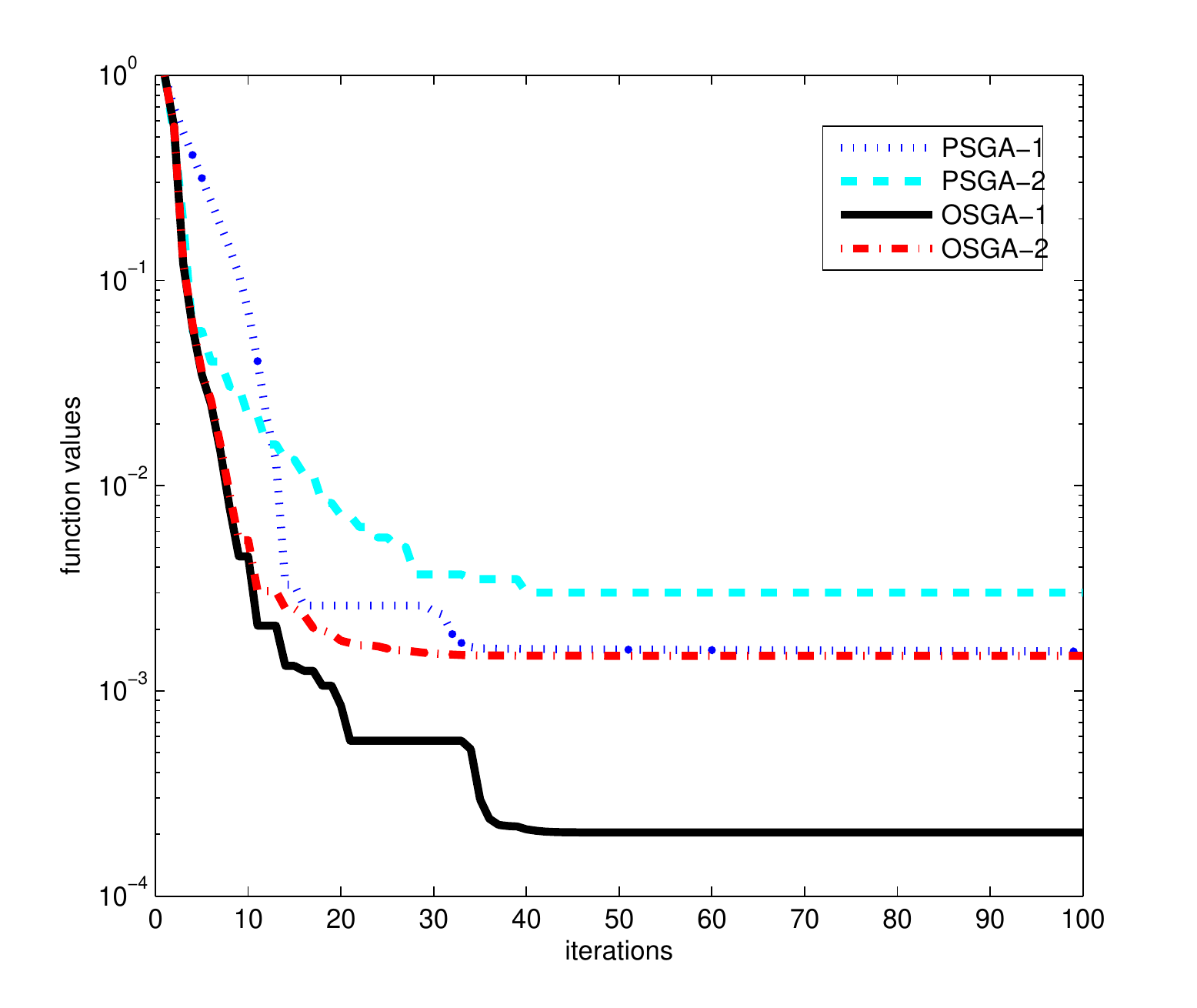}}%
\qquad
\subfloat[][L1L1R, $n=5000$]{\includegraphics[width=6.1cm]{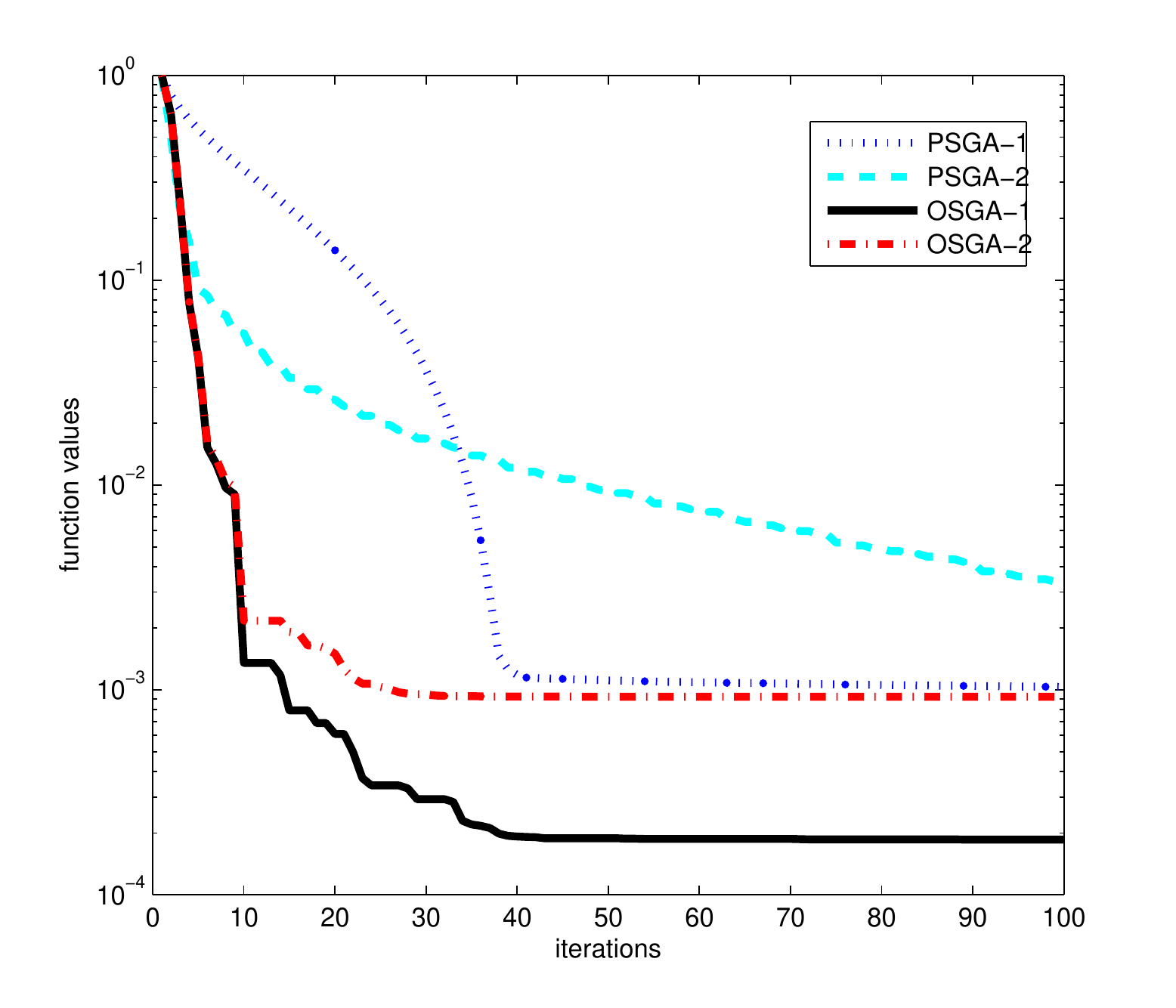}}

\caption{A comparison among PSGA-1, PSGA-2, OSGA-1, and OSGA-2 for solving L1L22R and L1L1R, where the algorithms were stopped after 100 iterations.}
\end{figure}

\vspace{5mm}
In Figure 1, Subfigures (a) and (b) show that OSGA-1 and OSGA-2  substantially outperform PSGA-1 and PSGA-2 with respect to the relative error of function values $\delta_k$ (\ref{e.delta}), however, they need more running time. In this case OSGA-1 and OSGA-2 are competitive, while OSGA-1 performs better. In Figure 1, Subfigure (c) shows that OSGA-2 produces the best results and the others are competitive. Subfigure (d) of Figure 1 demonstrates that OSGA-1 attains the best results and SPGA-2 and OSGA-2 are competitive but much better than SPGA-1.
In Figure 2, Subfigures (a) and (b) show that OSGA-1 and OSGA-2 are comparable but much better that PSGA-1 and PSGA-2. Subfigures (c) and (s) show that the best results produced by OSGA-1 and OSGA-2, respectively.

\subsection{Image deblurring/denoising}
Image deblurring/denoising is one of the fundamental tasks in the 
context of digital imaging processing, aiming at recovering an image from
a blurred/noisy observation. The problem is typically modelled as linear
inverse problem 
\begin{equation}\label{e.inv}
y = Ax + \omega,~~~ x \in \mathcal{V},
\end{equation}
where $\mathcal{V}$ is a finite-dimensional vector space, $A$ is a blurring linear operator, $x$ is a clean image, $y$ is an 
observation, and $\omega$ is either Gaussian or impulsive noise. 

The system of equation (\ref{e.inv}) is usually underdetermined and ill-conditioned, 
and $\omega$ is not commonly available, so it is not possible to solve it directly, see \cite{NeuI}. 
Hence the solution is generally approximated by an optimization problem of the form  
\begin{equation}\label{e.least}
\begin{array}{ll}
\min &~~ \D \frac{1}{2} \| Ax - b \|_2^2 + \lambda \varphi(x)\\
\st  &~~ x \in \mathcal{V},
\end{array}
\end{equation}
or
\begin{equation}\label{e.l1}
\begin{array}{ll}
\min &~~ \D \| Ax - b \|_1 + \lambda \varphi(x)\\
\st  &~~ x \in \mathcal{V},
\end{array}
\end{equation}
where $\varphi$ is a smooth or 
nonsmooth regularizer such as $\varphi(x) = \frac{1}{2} \|x\|_2^2$,
$\varphi(x) = \|x\|_1$, $\varphi(x) = \|x\|_{ITV}$, and $\varphi(x) = \|x\|_{ATV}$. Among the various regularizers, 
the total variation is much more popular due to its strong edge 
preserving feature. Two important types of the total variation, namely isotropic
and anisotropic, see \cite{ChaCCNP}, are defined for $x \in \Rz^{m\times n}$ by
\[
    \bary{lll}
    \|x\|_{ITV} &=& \sum_i^{m-1} \sum_j^{n-1} \sqrt{(x_{i+1,j} - x_{i,j})^2+(x_{i,j+1} - x_{i,j})^2 }\\
            &+& \sum_i^{m-1} |X_{i+1,n} - Xx_{i,n}| + \sum_i^{n-1} |x_{m,j+1} - x_{m,j}|
    \eary
\]
and
    \[
    \bary{lll}
    \|x\|_{ATV} &=& \sum_i^{m-1} \sum_j^{n-1} \{|x_{i+1,j} - x_{i,j}| + |x_{i,j+1} - x_{i,j}| \}\\
            &+& \sum_i^{m-1} |x_{i+1,n} - x_{i,n}| + \sum_i^{n-1} |x_{m,j+1} - x_{m,j}|,
    \eary
\]
respectively. 

The common drawback of the unconstrained problem (\ref{e.least}) is 
that it usually gives a solution outside of the dynamic range of the image, which is either $[0,1]$ or $[0,255]$ for 8-bit gray-scale images. 
Hence one has to project the unconstrained solution to the dynamic range of the image. However, the quality of the projected images is not always acceptable. As a result, it is worth to solve a 
bound-constrained problem of the form (\ref{e.bcls}) in place of the 
unconstrained problem (\ref{e.least}), where the bounds are defined by the dynamic range of the images, see \cite{BecT,ChaTY1,WooY}. 

The comparison concerning the quality of the recovered image is made via the so-called 
peak signal-to-noise ratio (PSNR) defined by 
\lbeq{e.psnr}
\mathrm{PSNR} = 20 \log_{10} \left( \frac{\sqrt{mn}}{\|x - x_t\|_F} \right)
\eeq
and the improvement in signal-to-noise ratio (ISNR) defined by
\lbeq{e.isnr}
\mathrm{ISNR} = 20 \log_{10} \left( \frac{\|y - x_t\|_F}{\|x - x_t\|_F} \right), 
\eeq
where $\|\cdot\|_F$ is the Frobenius norm, $x_t$ denotes the $m \times n$ true image,
$y$ is the observed image, and pixel values are in $[0, 1]$.

\vspace{5mm}
\subsubsection{Experiment with $l_2^2$ isotropic total variation}
We here consider the image restoration from a blurred/noisy observation using 
the model (\ref{e.bcls}) equipped with the isotropic total variation regularizer. 
We employ OSGA, MFISTA (a monotone version of 
FISTA proposed by {\sc Beck \& Teboulle} in \cite{BecT}), ADMM (an alternating 
direction method proposed by {\sc Chan} et al. in \cite{ChaTY1}), and a projected 
subgradient algorithms PSGA (with nonsummable diminishing step-size, see \cite{BoyXM}). 
In our implementation we use the original code of MFISTA and ADMM provided by the authors, with minor adaptations to solve the problem form (\ref{e.bcls}) 
and to stop in a fixed number of iterations. 

We now restore the $512 \times 512$ blurred/noisy Barbara image. Let $y$ be a blurred/noisy version of this image generated by a $9 \times 9$ 
uniform blur and adding a Gaussian noise with zero mean and standard deviation 
set to $10^{-3/2}$. Our implementation shows that the algorithms are sensitive to the regularization parameter $\lambda$. 
Hence we consider three different regularization parameters 
$\lambda = 1 \times 10^{-2}$, $\lambda = 7 \times 10^{-3}$, and $\lambda = 4 \times 10^{-3}$. All algorithms are stopped after 50 iterations. 
The results of our implementation are summarized in Table 3 and Figures 3 and 4. 

\begin{table}[htbp]
\caption{Result summary for the $l_2^2$ isotropic total variation}
\begin{center}\footnotesize
\renewcommand{\arraystretch}{1.3}
\begin{tabular}{|l|l|l|l|l|l|}\hline
\multicolumn{1}{|l|}{} & \multicolumn{1}{l|}{{\bf $\lambda$}}
&\multicolumn{1}{l|}{{\bf PSGA}} & \multicolumn{1}{l|}{{\bf MFISTA}}
&\multicolumn{1}{l|}{{\bf ADMM}} & \multicolumn{1}{l|}{{\bf OSGA}} \\ 
\hline
{\bf PSNR}  &                    & 23.69 & 23.64 & 23.59 & 23.74\\
{\bf $f_b$} & $1 \times 10^{-2}$ & 1.6804e+2 & 1.6580e+2 & 1.6705e+2 & 1.6531e+2\\
{\bf Time } &                    & 2.15 & 30.55 & 2.05 & 5.60\\
\hline
{\bf PSNR}  &                    & 23.73 & 23.66 & 23.67 & 23.76\\
{\bf $f_b$} & $7 \times 10^{-3}$ & 1.5694e+2 & 1.5543e+2 & 1.5599e+2 & 1.5500e+2\\
{\bf Time } &                    & 2.09 & 31.31 & 2.12 & 5.39\\
\hline
{\bf PSNR}  &                    & 23.77 & 23.67 & 23.63 & 23.77\\
{\bf $f_b$} & $4 \times 10^{-3}$ & 1.4402e+2 & 1.4321e+2& 1.4329e+2 & 1.4294e+2\\
{\bf Time } &                    & 2.09 & 31.93 & 2.06 & 5.49\\
\hline
\end{tabular}
\end{center}
\end{table}

The results of Table 3 and Figure 3 show that function values, PSNR, and ISNR produced by the algorithms are sensitive to the regularization parameter $\lambda$. 
However, the function values are less sensitive.
Subfigures (a), (c), and (e) show that OSGA gives the best performance in terms of function values. According to subfigures (b), (d), and (f), 
the best ISNR is attained for $\lambda = 4 \times 10^{-3}$, and the algorithms
are comparable with each other, but OSGA outperforms the others slightly. Figure 4 illustrates
the resulting deblurred images for $\lambda = 4 \times 10^{-3}$.

\begin{figure} \label{f.deb1}
\centering
\subfloat[][$\delta_k$ versus iterations, $\lambda = 1 \times 10^{-3}$]{\includegraphics[width=6.1cm]{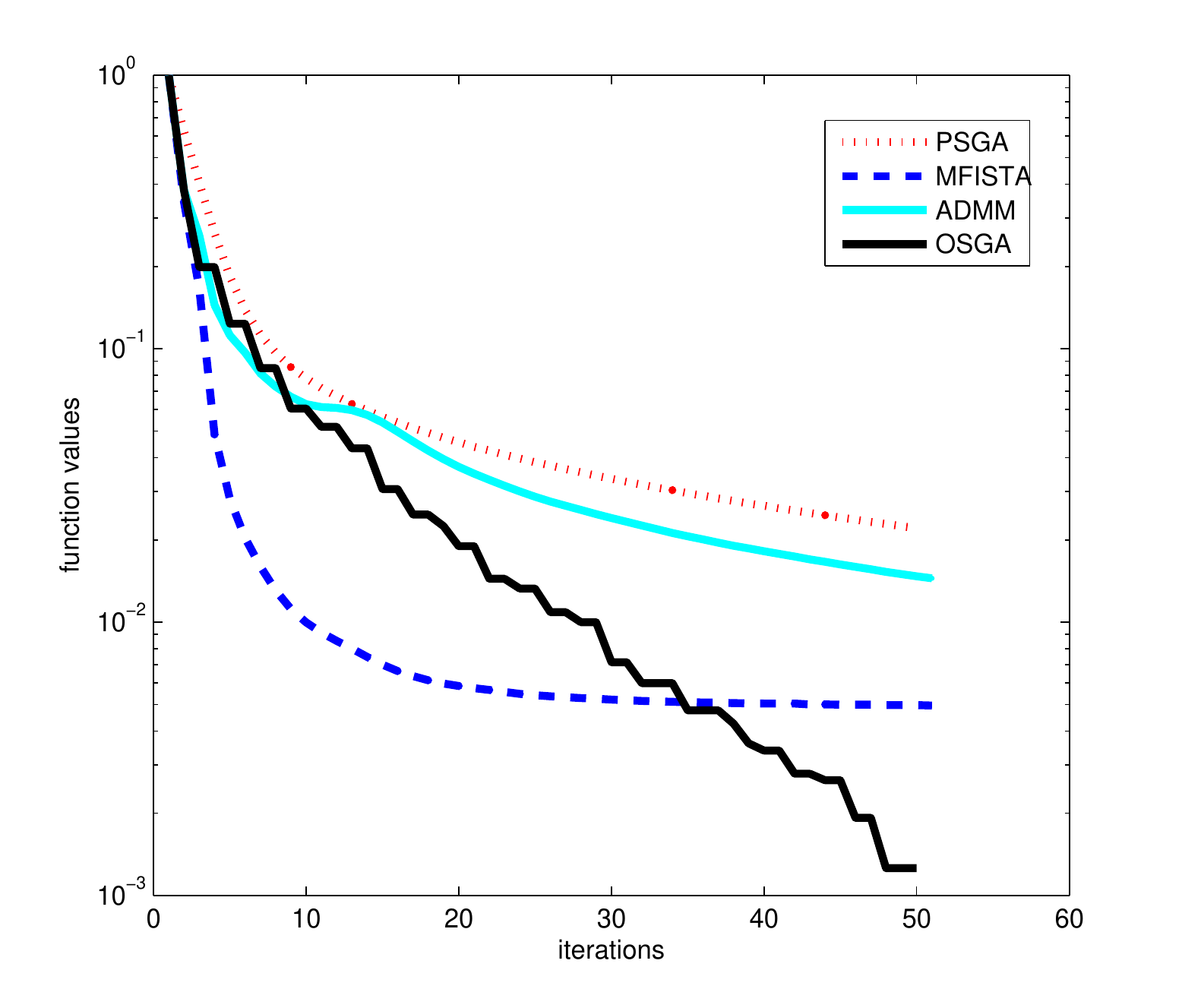}}%
\qquad
\subfloat[][ISNR versus iterations, $\lambda = 1 \times 10^{-3}$]{\includegraphics[width=6.1cm]{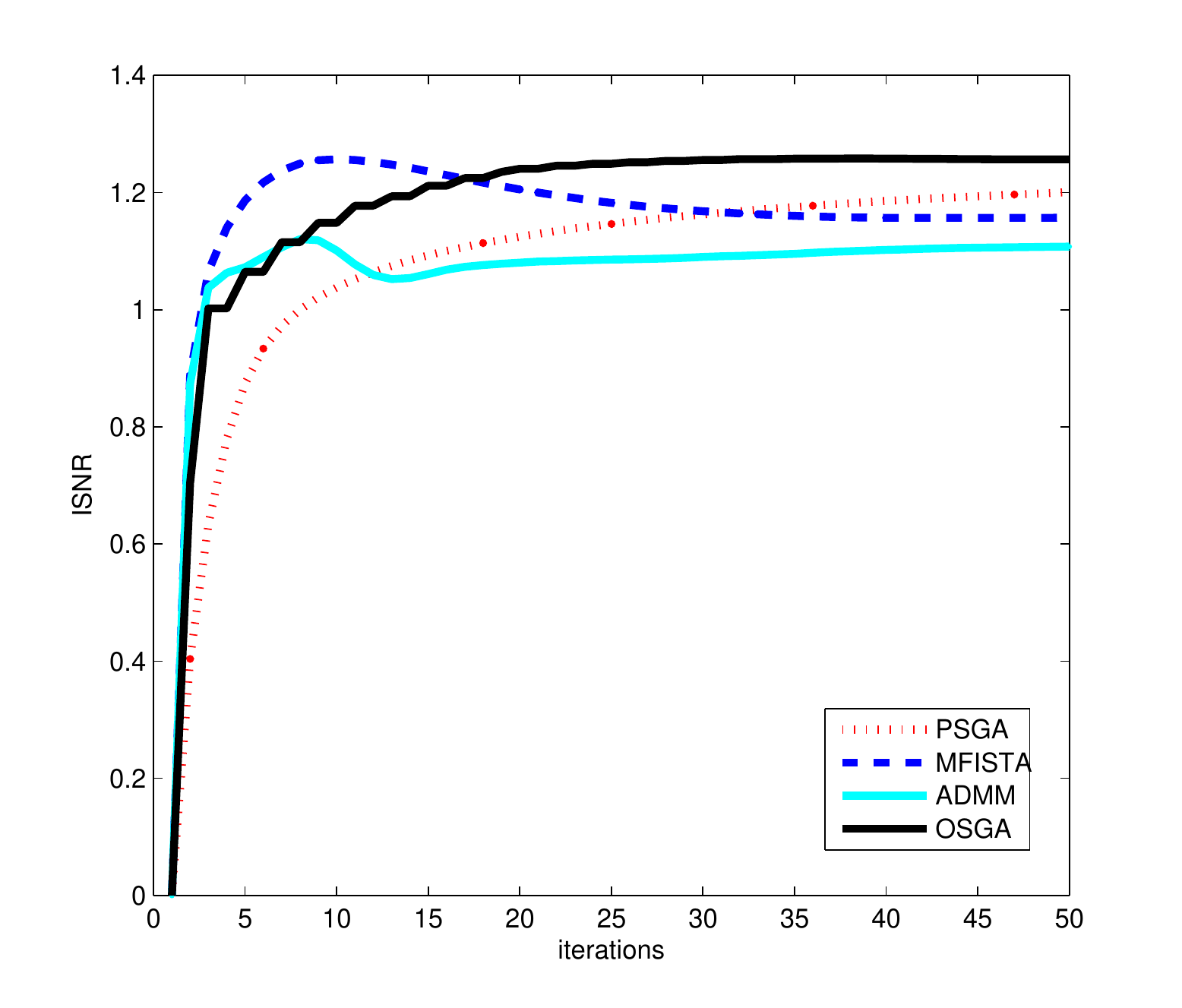}}
\qquad
\subfloat[][$\delta_k$ versus iterations, $\lambda = 6 \times 10^{-4}$]{\includegraphics[width=6.1cm]{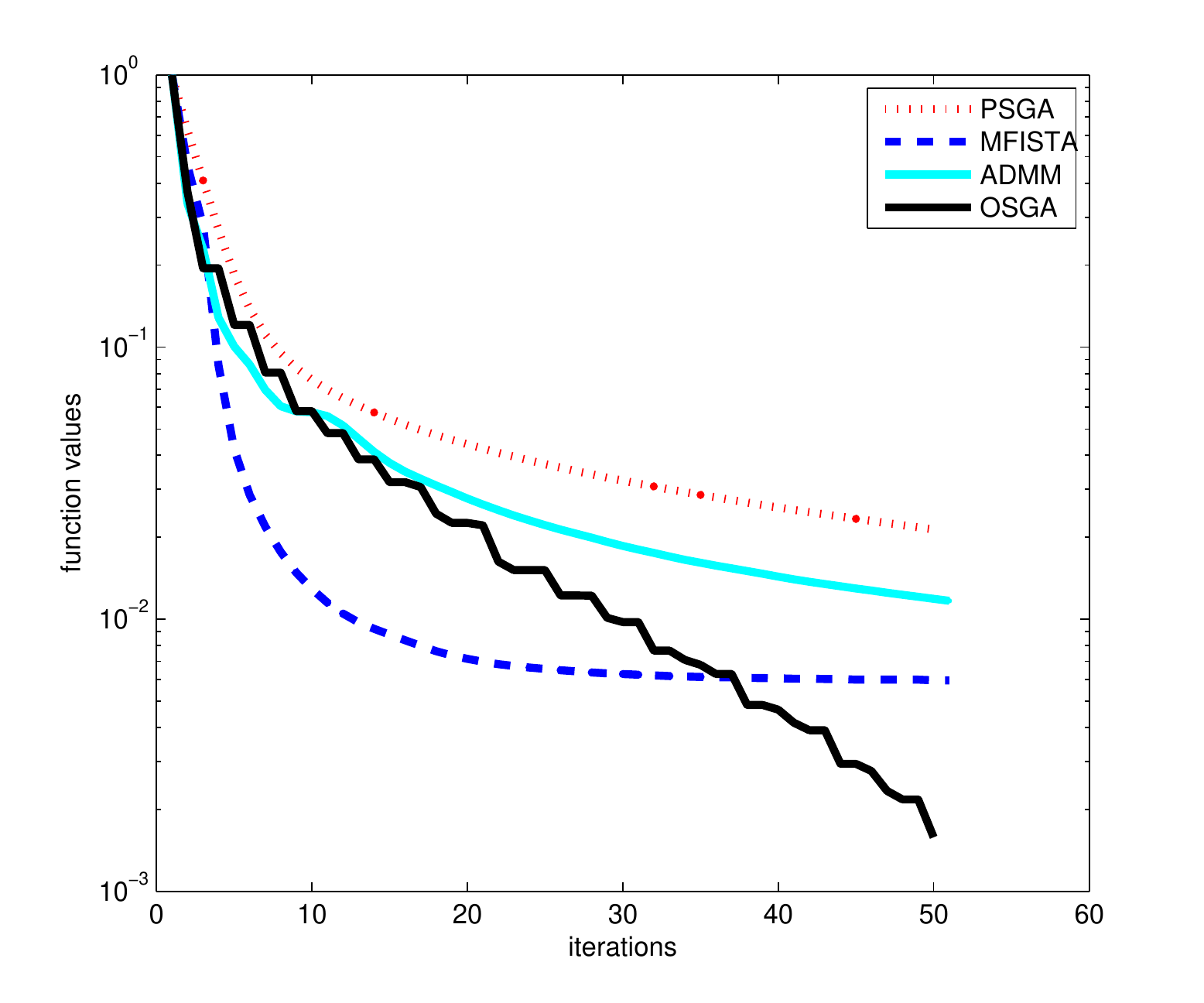}}%
\qquad
\subfloat[][ISNR versus iterations, $\lambda = 6 \times 10^{-4}$]{\includegraphics[width=6.1cm]{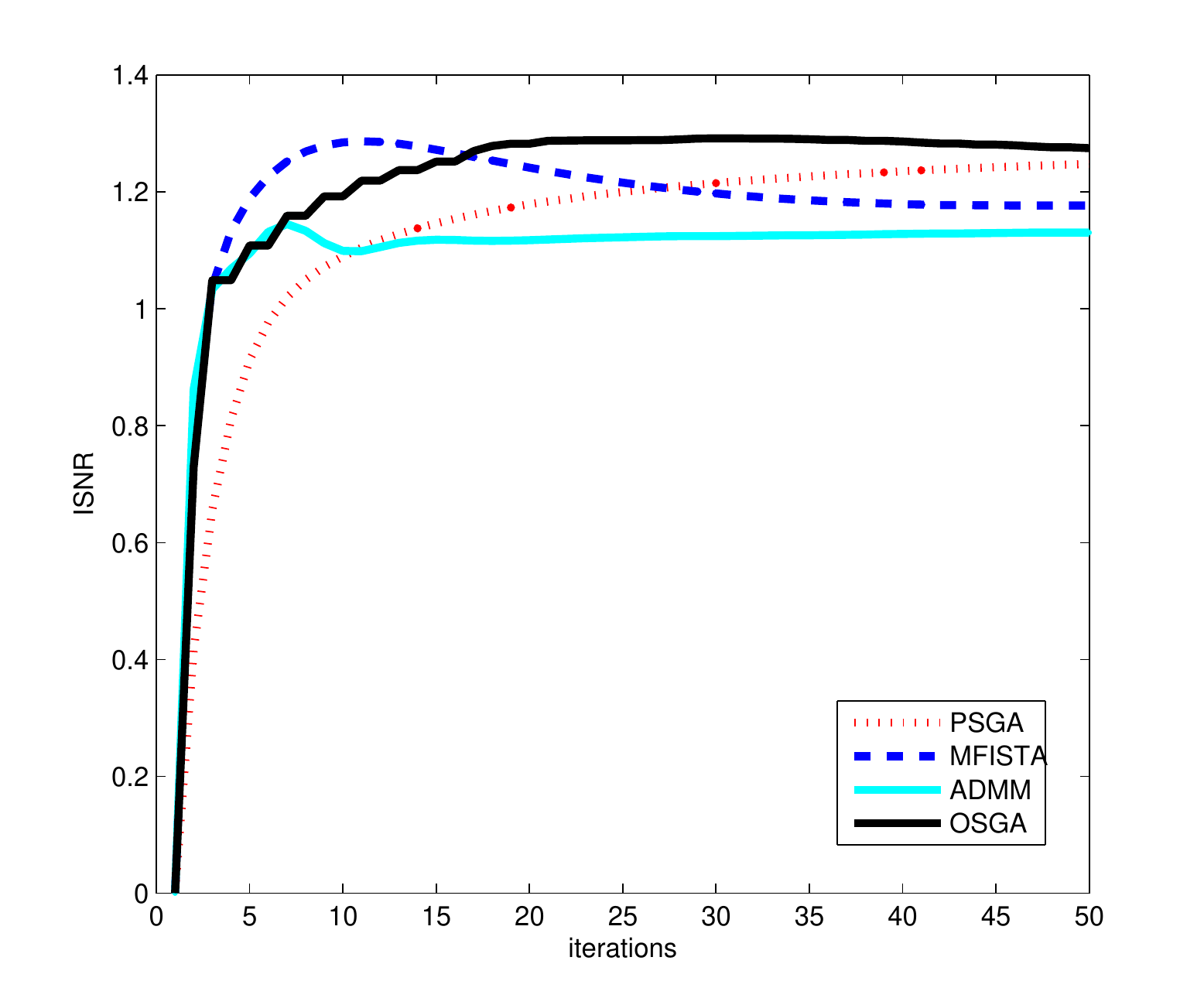}}
\qquad
\subfloat[][$\delta_k$ versus iterations, $\lambda = 4 \times 10^{-4}$]{\includegraphics[width=6.1cm]{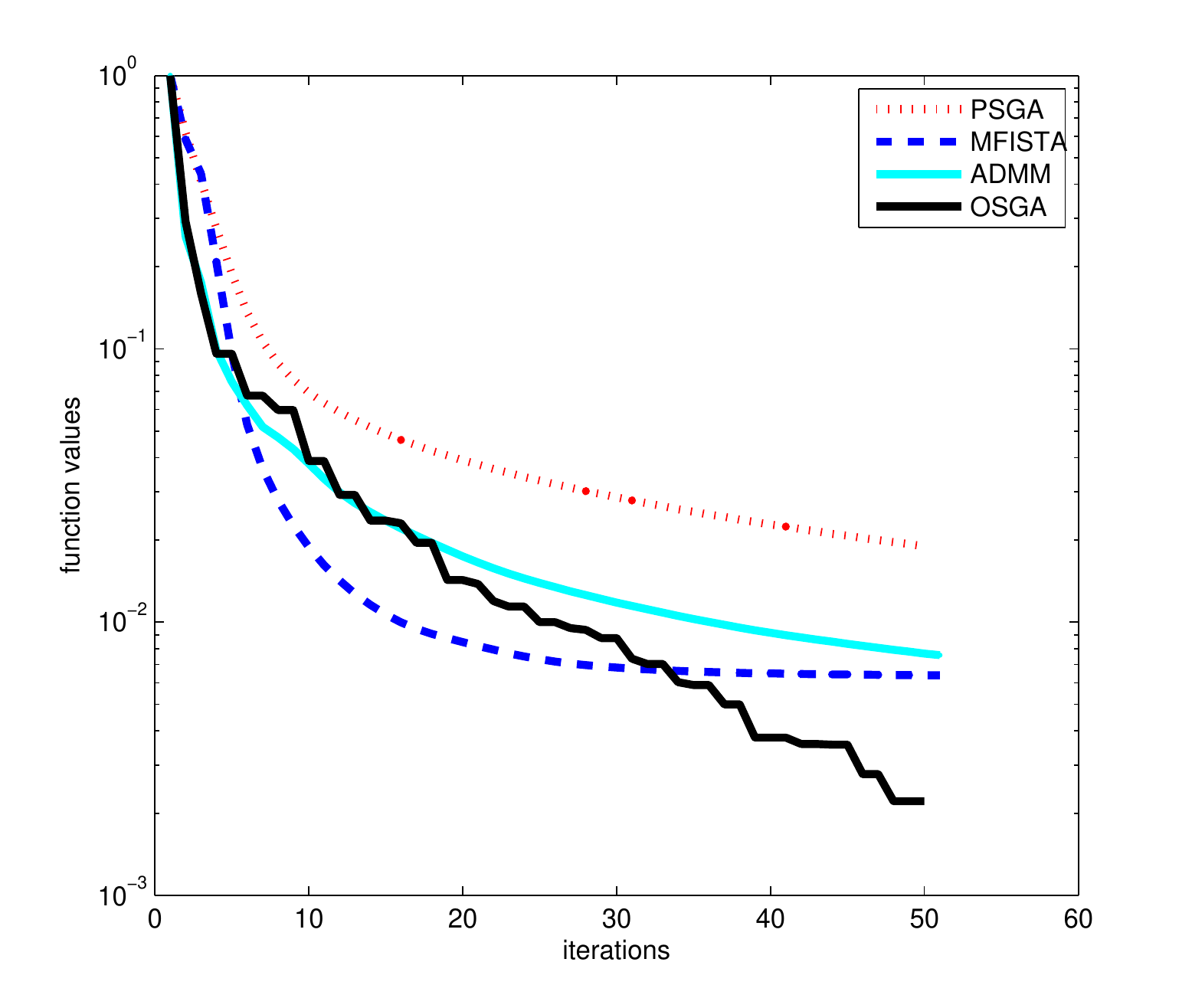}}%
\qquad
\subfloat[][ISNR versus iterations, $\lambda = 4 \times 10^{-4}$]{\includegraphics[width=6.1cm]{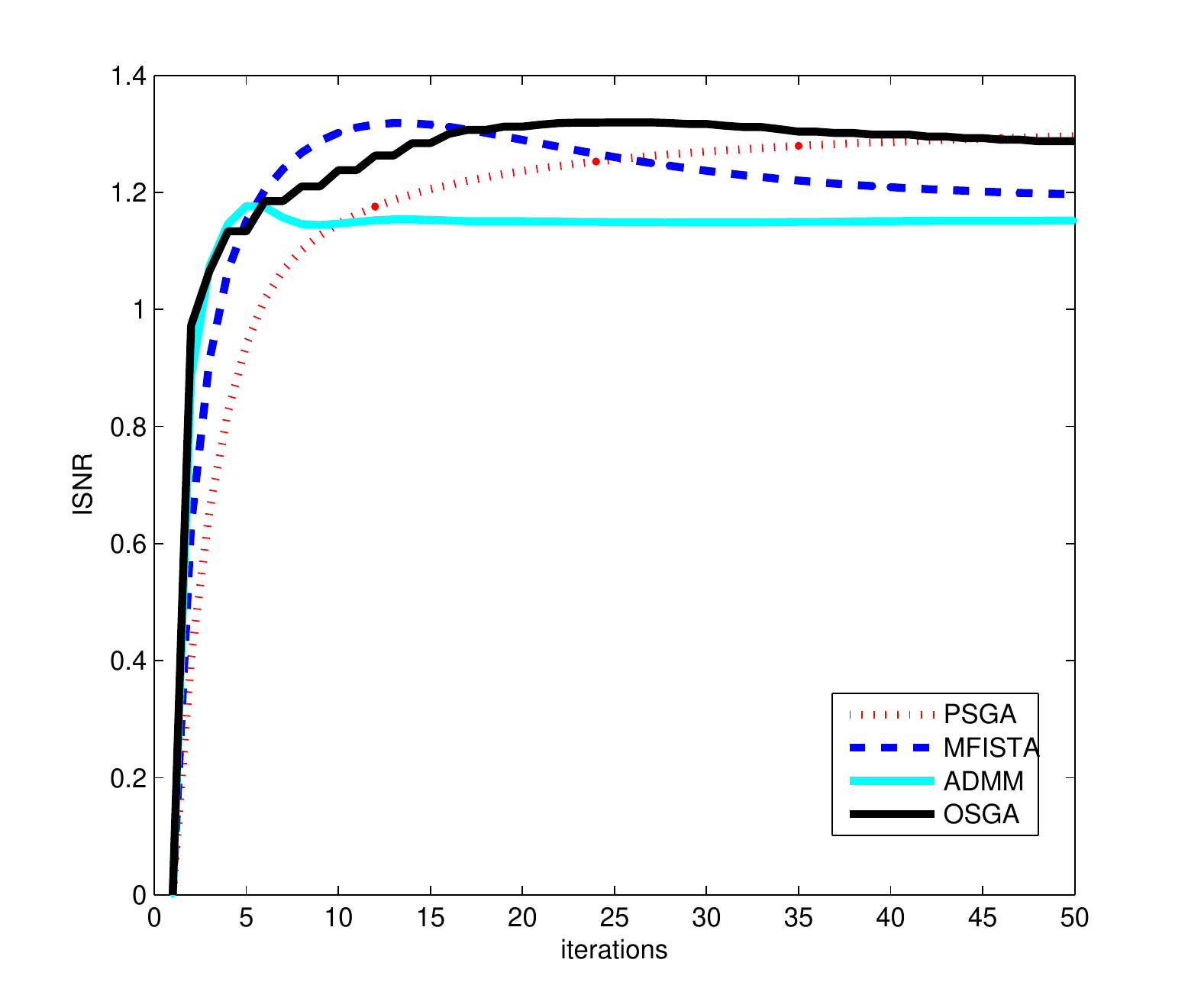}}

\caption{A comparison among PSGA, MFISTA, ADMM, and OSGA for deblurring the $512 \times 512$ Barbara image with the $9 \times 9$ uniform blur and the Gaussian noise with deviation $10^{-3/2}$. The algorithms were stopped after 50 iterations. The subfigures (a), (c) and, (e) display the relative error $\delta_k$ (\ref{e.delta}) of function values versus iterations, and (b), (d), and (f) show ISNR (\ref{e.isnr}) versus iterations.}
\end{figure}

\begin{figure} \label{f.deb2}
\centering
\subfloat[][Original image]{\includegraphics[width=6.1cm]{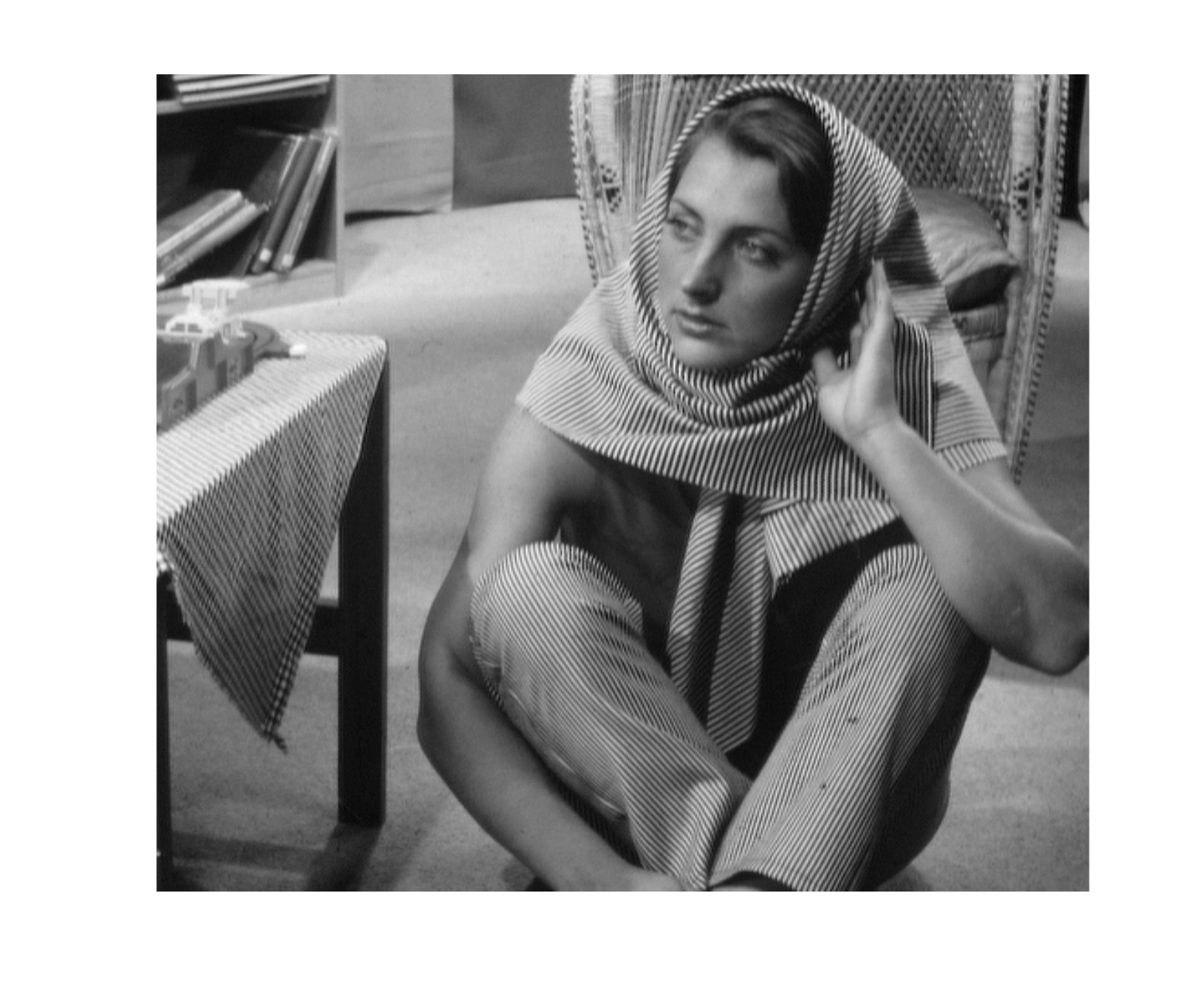}}%
\qquad
\subfloat[][Blurred/noisy image]{\includegraphics[width=6.1cm]{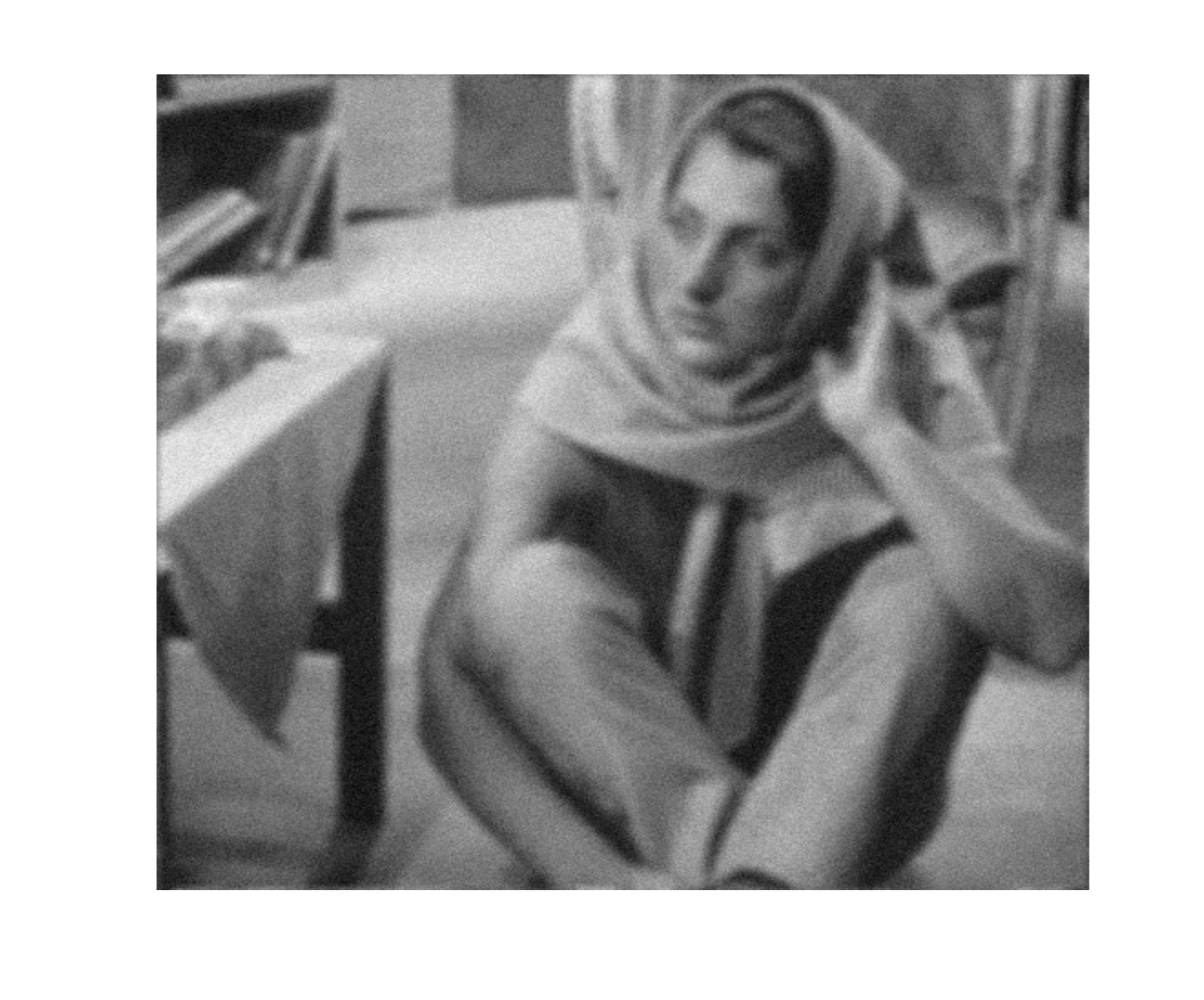}}
\qquad
\subfloat[][PSGA: $f = 1.4402e+2, \mathrm{PSNR} = 23.77, \mathrm{T} = 2.09$]{\includegraphics[width=6.1cm]{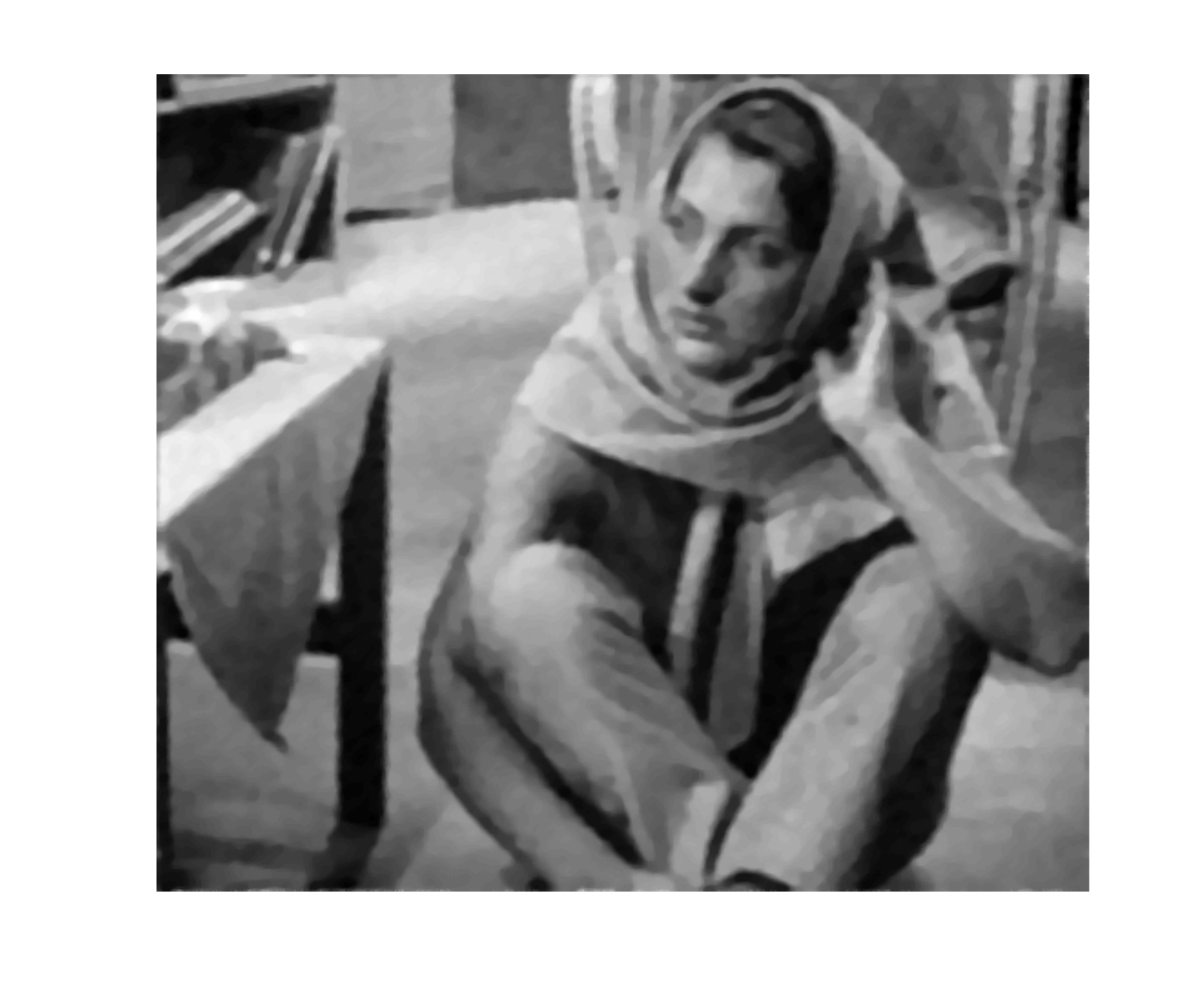}}%
\qquad
\subfloat[][MFISTA: $f = 1.4321e+2, \mathrm{PSNR} = 23.67, \mathrm{T} = 31.93$]{\includegraphics[width=6.1cm]{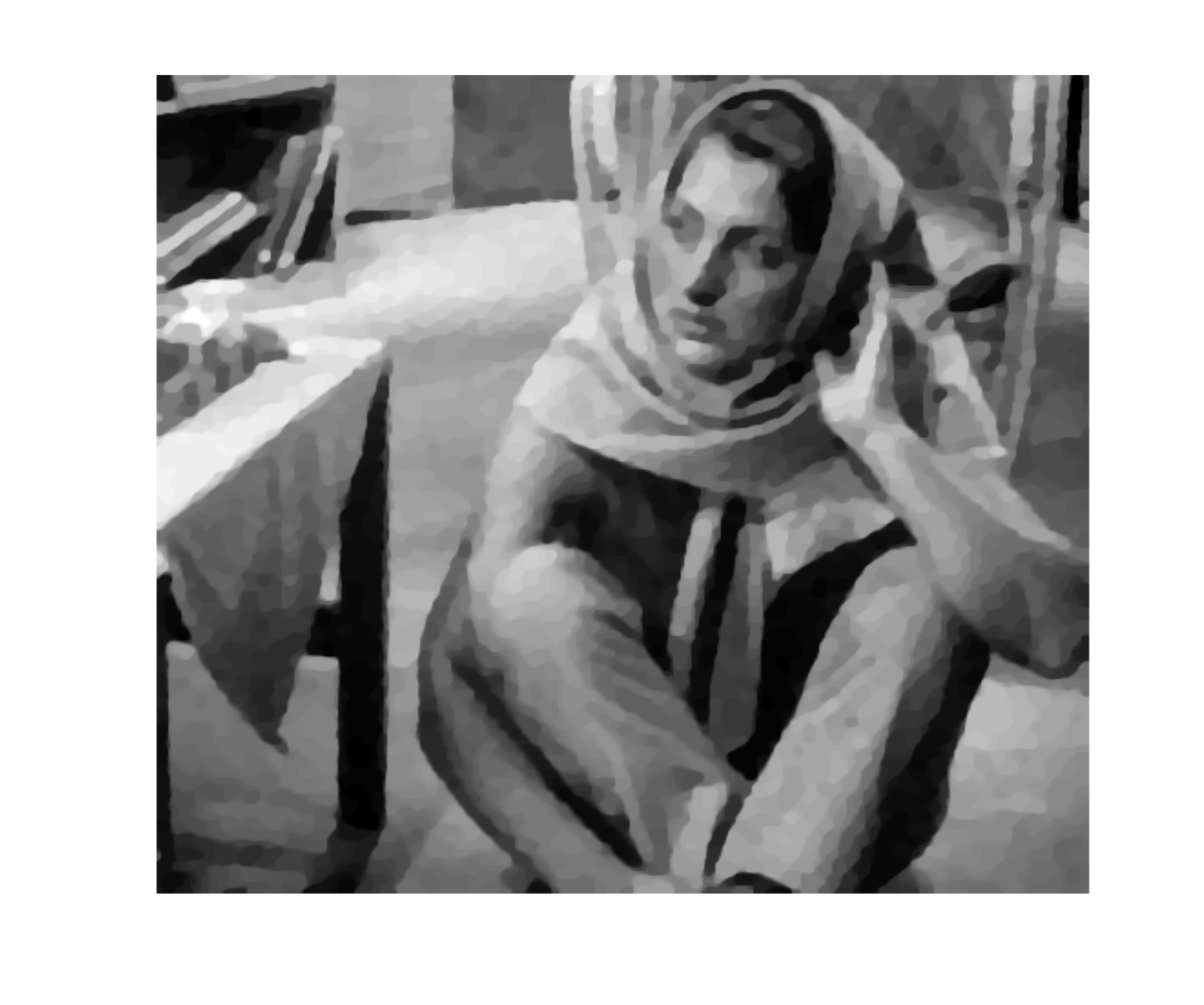}}
\qquad
\subfloat[][ADMM: $f = 1.4329e+2, \mathrm{PSNR} = 23.63, \mathrm{T} = 2.06$]{\includegraphics[width=6.1cm]{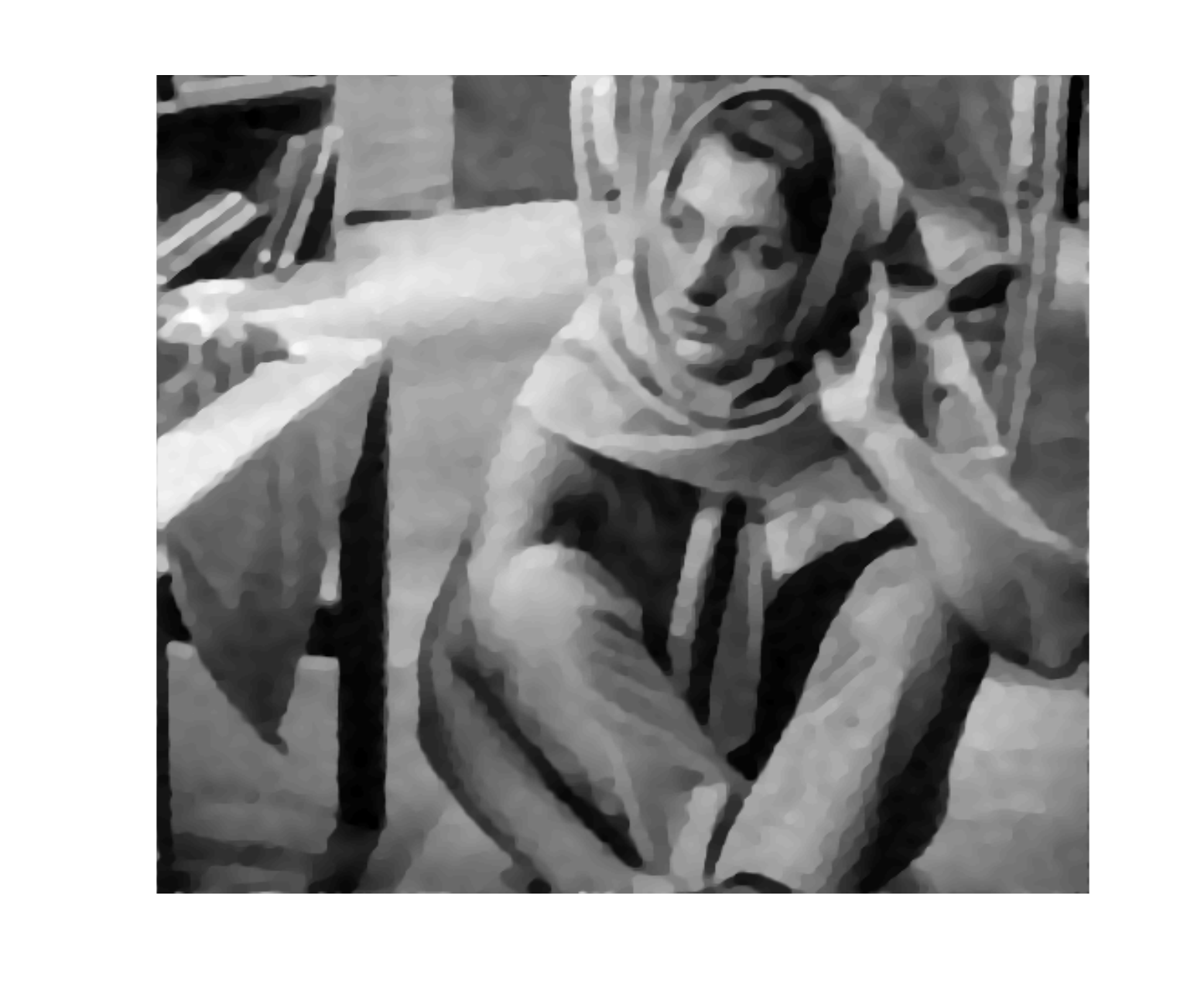}}%
\qquad
\subfloat[][OSGA: $f = 1.4294e+2, \mathrm{PSNR} = 23.77, \mathrm{T} = 5.49$]{\includegraphics[width=6.1cm]{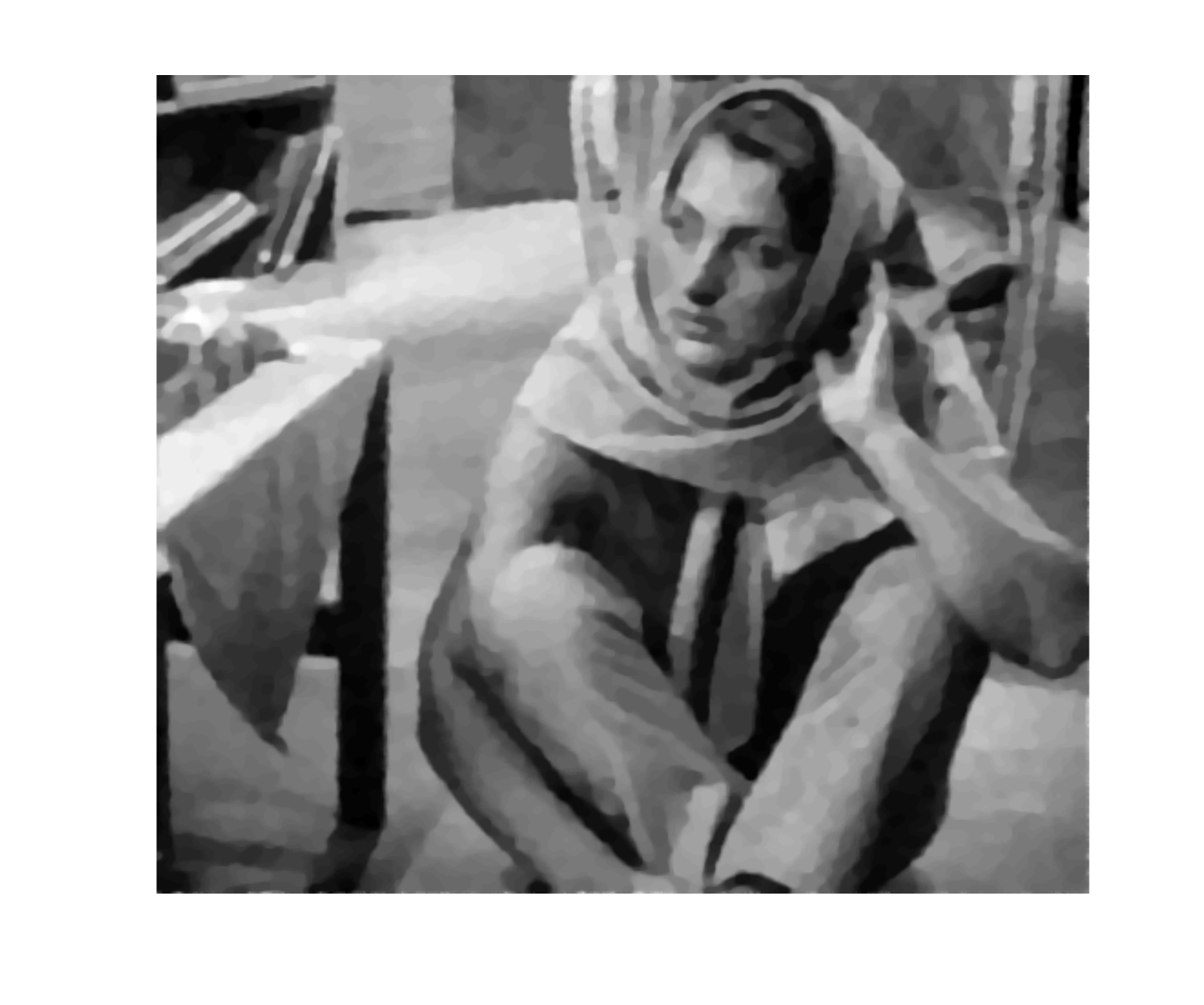}}
\caption{Deblurring of the $512 \times 512$ Barbara image with the $9 \times 9$ uniform blur and the Gaussian noise with deviation $10^{-3/2}$ by PSGA, MFISTA, ADMM, and OSGA with the regularization parameter $\lambda = 4 \times 10^{-3}$. The algorithms were stopped after 50 iterations.}%
\end{figure}

\vspace{5mm}
\subsubsection{Experiment with $l_1$ isotropic total variation}
In this subsection we study the image restoration from a blurred/noisy observation using 
the model (\ref{e.bcl1}) equipped with the isotropic total variation regularizer. The 
optimization problem is solved by DRPD1, DRPD2 (Douglas-Rachford primal-dual algorithms proposed by {\sc Bo? \& Hendrich} in \cite{BotH}), ADMM, and OSGA. 

Here we consider recovering the $256 \times 256$ blurred/noisy Fingerprint image from a
blurred/noisy image constructed by a $7 \times 7$ Gaussian kernel with 
standard deviation 5 and salt-and-pepper impulsive noise with the level $40 \%$. 
The algorithms are stopped after 50 iterations. Three different regularization parameters 
$\lambda = 3 \times 10^{-1}$, $\lambda = 1 \times 10^{-1}$, and $\lambda = 8 \times 10^{-2}$
are considered. The results are presented in Table 4 and Figures 5 and 6. 

The results of Figure 5
demonstrates that the algorithms are sensitive to the regularization parameter. For example, ADMM get the best function value and an acceptable ISNR compared with the others for $\lambda = 3 \times 10^{-1}$, however, for $\lambda = 1 \times 10^{-1}$ and $\lambda = 8 \times 10^{-2}$, it behaves worse than the others. In the sense of ISNR, DRPD1, DRPD2, and OSGA are competitive, but OSGA get the better results. The resulting images for $\lambda = 8 \times 10^{-2}$ are illustrated in Figure 6, demonstrating that ADMM could not recover the image properly, however, DRPD1, DRPD2, and OSGA reconstruct acceptable approximations, where OSGA attains the best PSNR.

\begin{table}[!htbp]
\caption{Results for the $l_1$ isotropic total variation}
\begin{center}\footnotesize
\renewcommand{\arraystretch}{1.3}
\begin{tabular}{|l|l|l|l|l|l|}\hline
\multicolumn{1}{|l|}{} & \multicolumn{1}{l|}{{\bf $\lambda$}}
&\multicolumn{1}{l|}{{\bf DRPD-1}} & \multicolumn{1}{l|}{{\bf DRPD-2}}
&\multicolumn{1}{l|}{{\bf ADMM}} & \multicolumn{1}{l|}{{\bf OSGA}} \\ 
\hline
{\bf PSNR}  &                    & 17.95 & 17.61 & 18.47 & 18.67\\
{\bf $f_b$} & $3 \times 10^{-1}$ & 1.4581e+4 & 1.4867e+4 & 1.4476e+4 & 1.4564e+4\\
{\bf Time } &                    & 1.04 & 0.67 & 0.81 & 6.11\\
\hline
{\bf PSNR}  &                    & 21.87 & 21.23 & 20.10 & 22.05\\
{\bf $f_b$} & $1 \times 10^{-1}$ & 1.3571e+4 & 1.3635e+4 & 1.3799e+4 & 1.3612e+4\\
{\bf Time } &                    & 1.28 & 0.83 & 0.76 & 6.01\\
\hline
{\bf PSNR}  &                    & 22.26 & 21.56 & 18.67 & 22.46\\
{\bf $f_b$} & $8 \times 10^{-2}$ & 1.3519e+4 & 1.3573e+4 & 1.3879e+4 & 1.3582e+4\\
{\bf Time } &                    & 1.00 & 0.65 & 0.79 & 6.11\\
\hline
\end{tabular}
\end{center}
\end{table}

\begin{figure} \label{f.deb3}
\centering
\subfloat[][$\delta_k$ versus iterations, $\lambda = 2 \times 10^{-1}$]{\includegraphics[width=6.1cm]{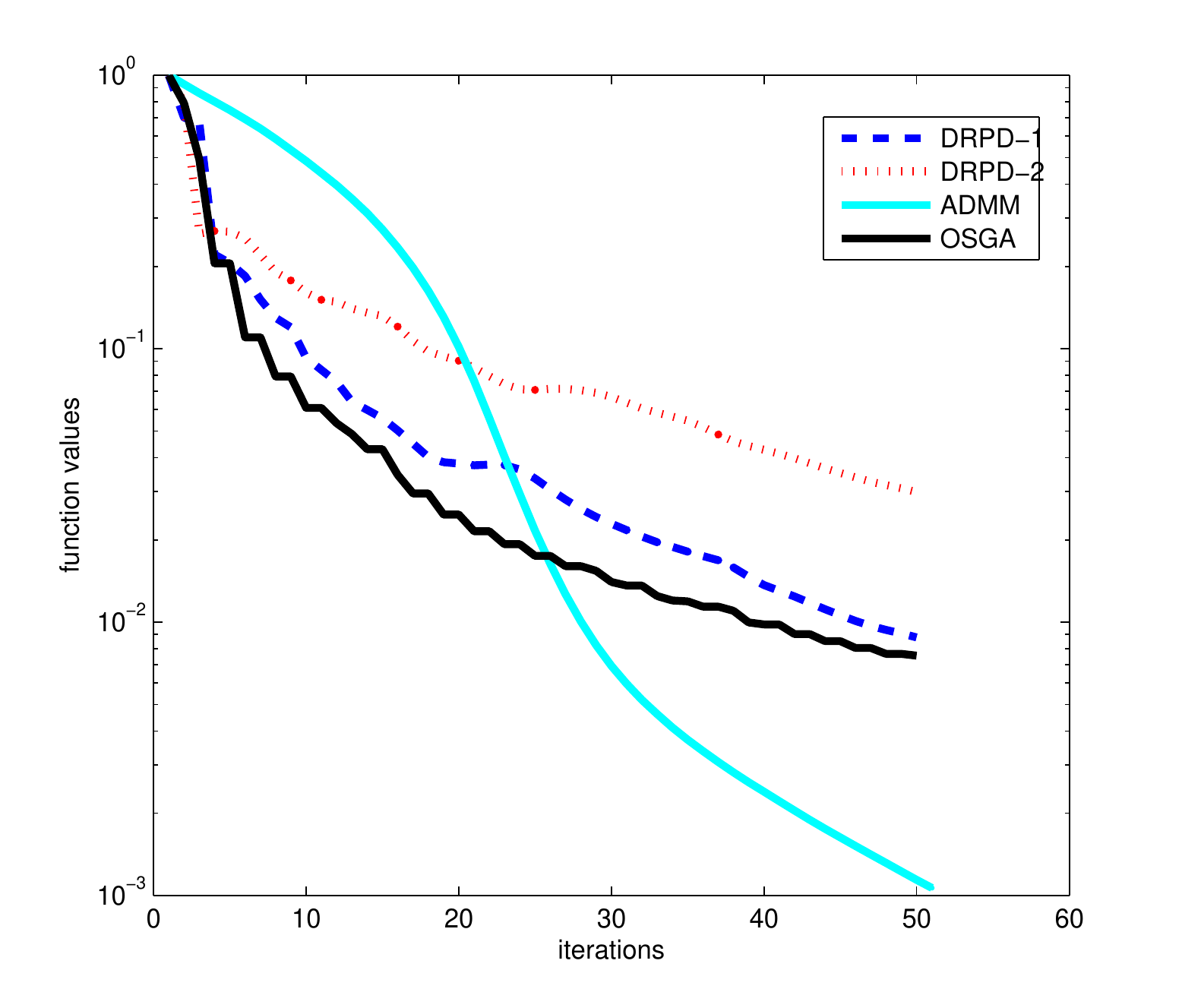}}%
\qquad
\subfloat[][ISNR versus iterations, $\lambda = 2 \times 10^{-1}$]{\includegraphics[width=6.1cm]{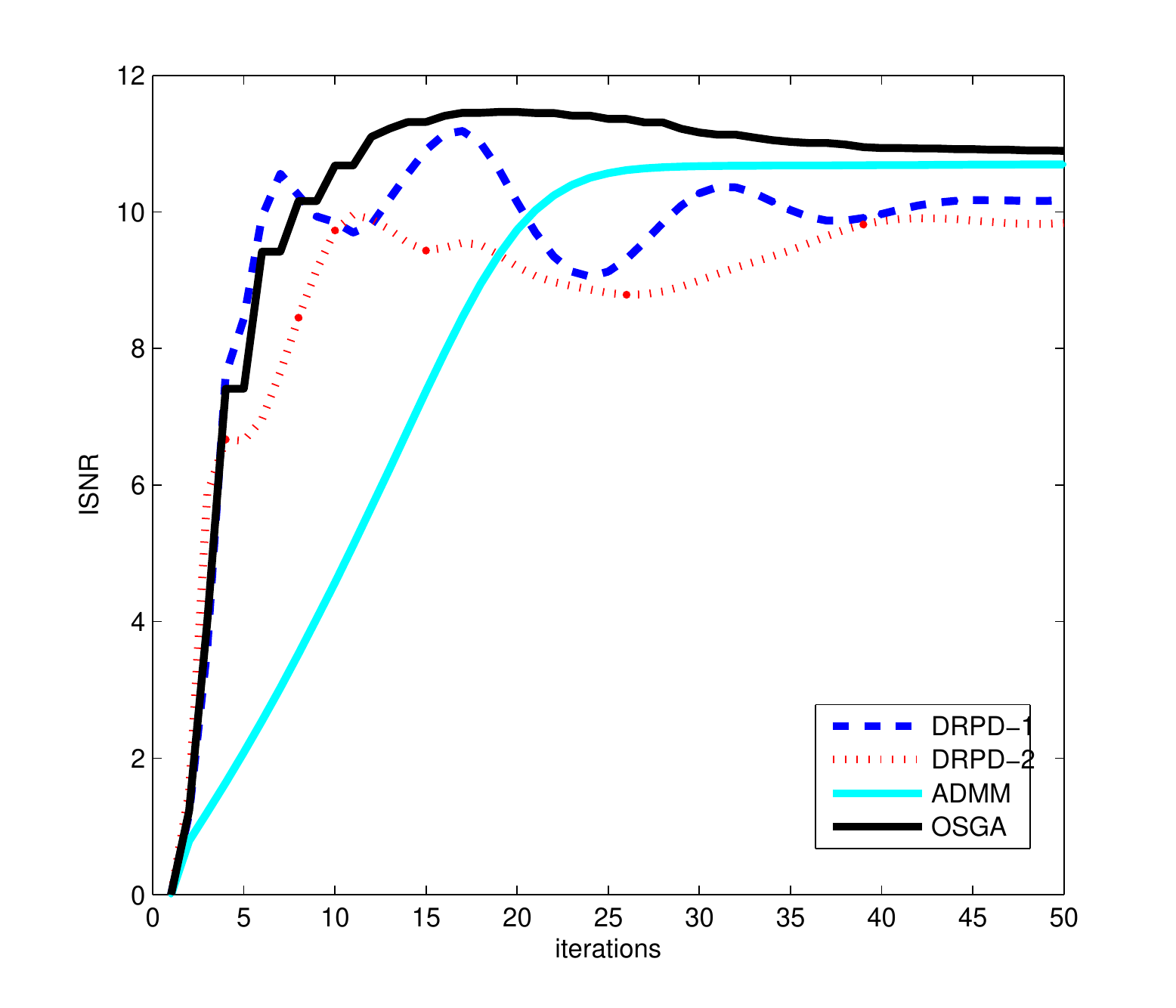}}
\qquad
\subfloat[][$\delta_k$ versus iterations, $\lambda = 8 \times 10^{-2}$]{\includegraphics[width=6.1cm]{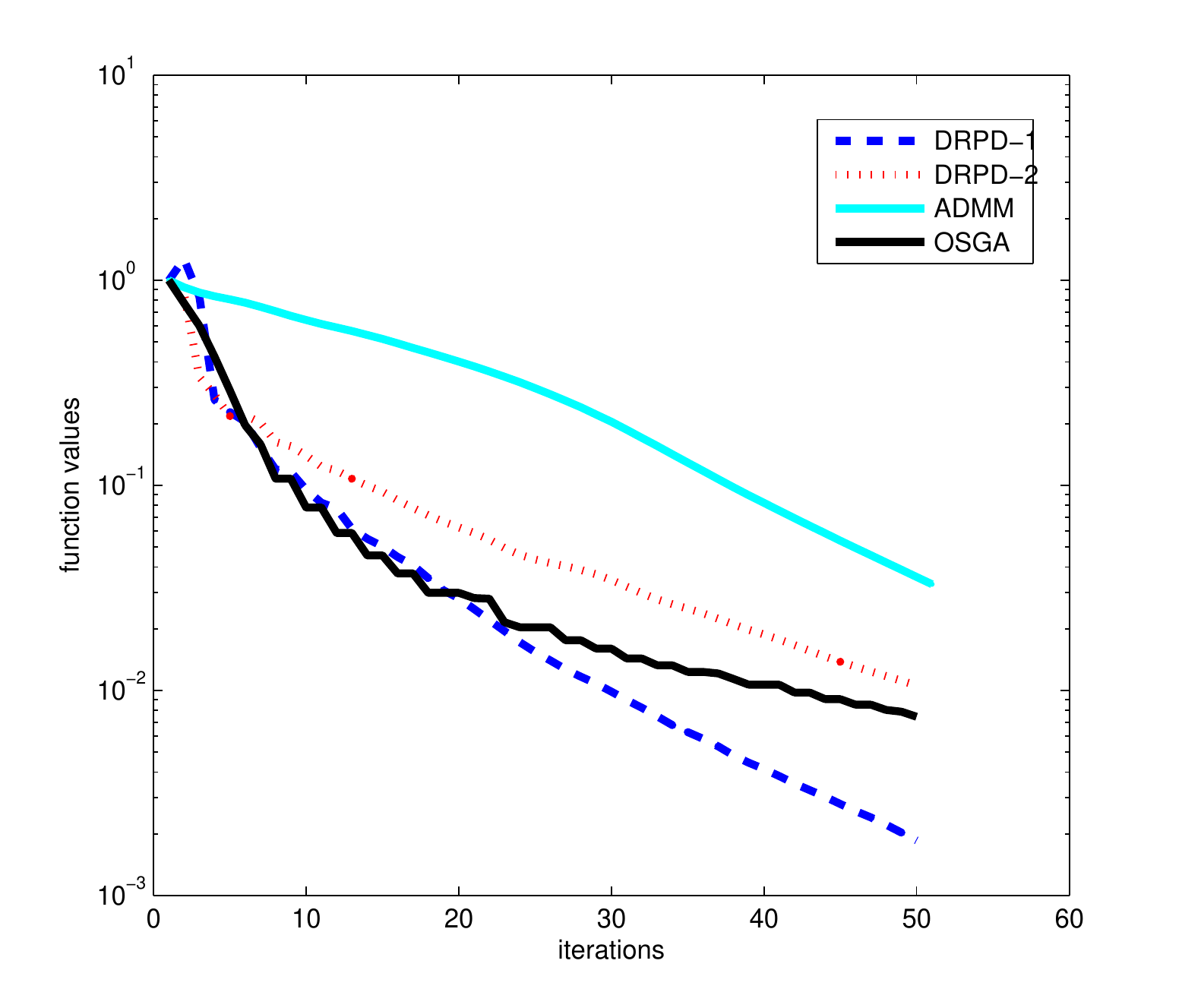}}%
\qquad
\subfloat[][ISNR versus iterations, $\lambda = 8 \times 10^{-2}$]{\includegraphics[width=6.1cm]{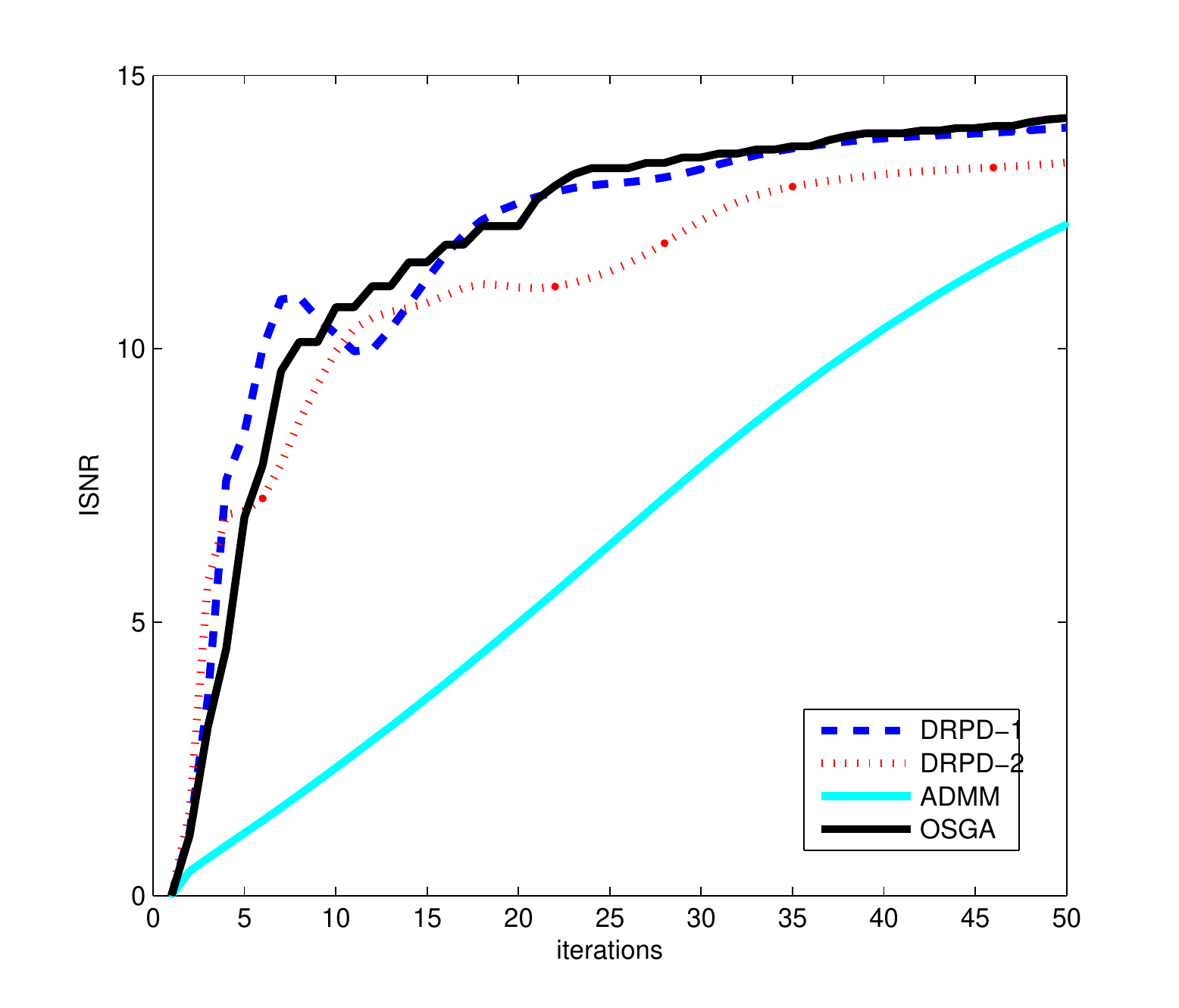}}
\qquad
\subfloat[][$\delta_k$ versus iterations, $\lambda = 6 \times 10^{-2}$]{\includegraphics[width=6.1cm]{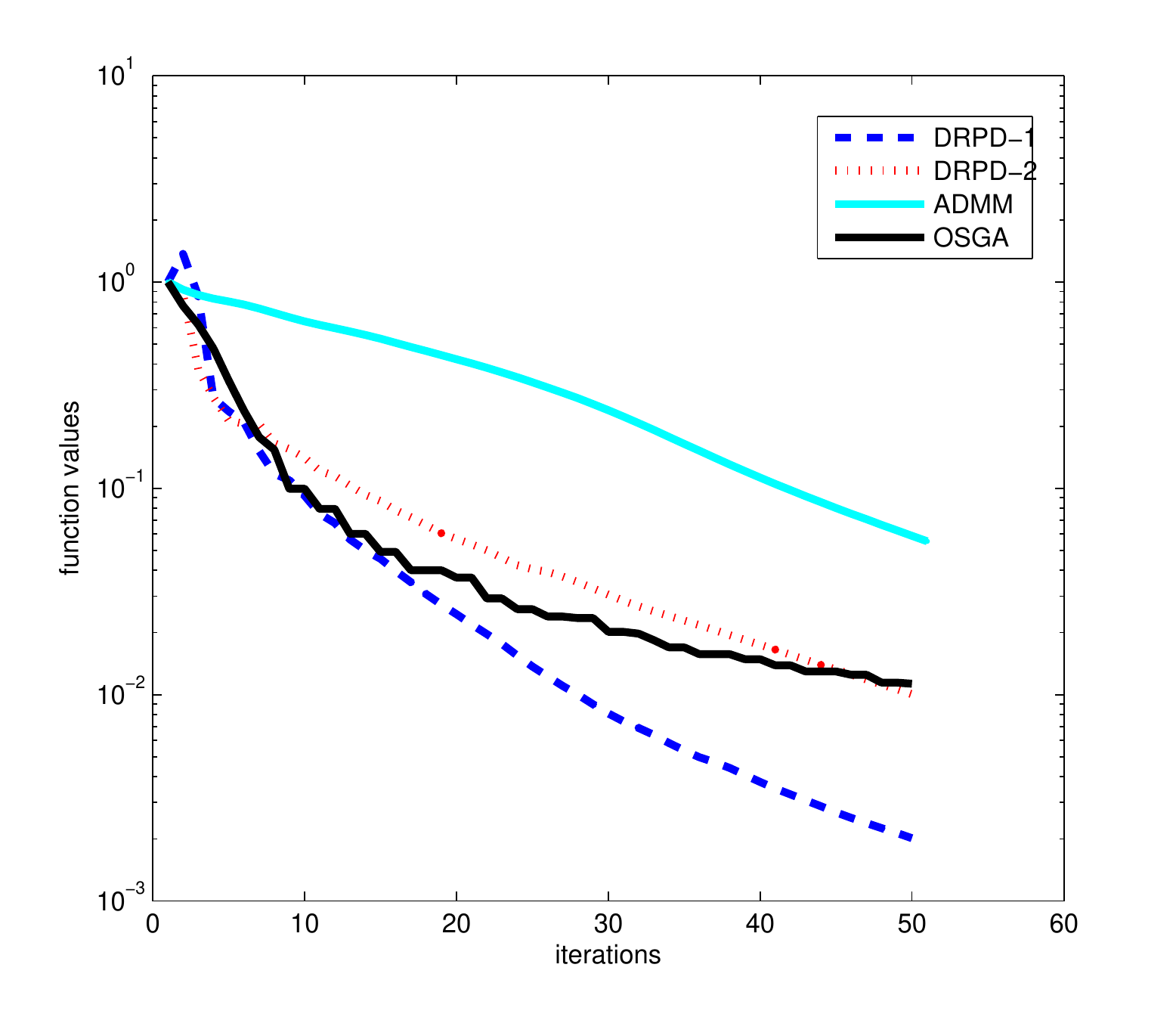}}%
\qquad
\subfloat[][ISNR versus iterations, $\lambda = 6 \times 10^{-2}$]{\includegraphics[width=6.1cm]{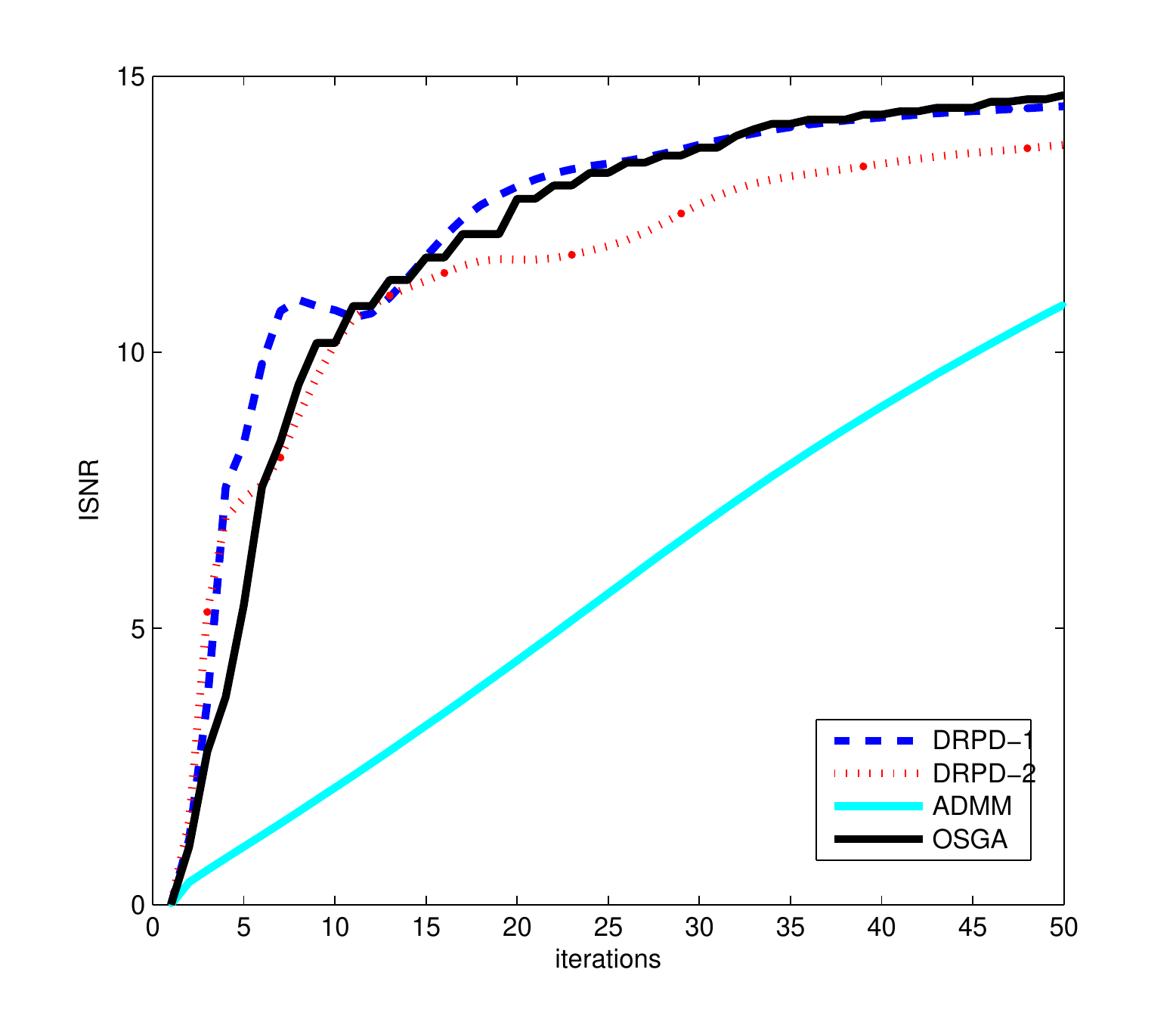}}

\caption{A comparison among DRPD1, DRPD2, ADMM, and OSGA for deblurring the $256 \times 256$ Fingerprint image with the various regularization parameter $\lambda$. The blurred/noisy image was constructed by the $7 \times 7$ Gaussian kernel with standard deviation 5 and salt-and-pepper impulsive noise with the level $50 \%$. The algorithms were stopped after 50 iterations. Subfigures (a), (c), and (e) display the relative error $\delta_k$ (\ref{e.delta}) of function values versus iterations, and (b), (d), and (f) show ISNR (\ref{e.isnr}) versus iterations.}
\end{figure}

\begin{figure} \label{f.deb4}
\centering
\subfloat[][Original image]{\includegraphics[width=6.1cm]{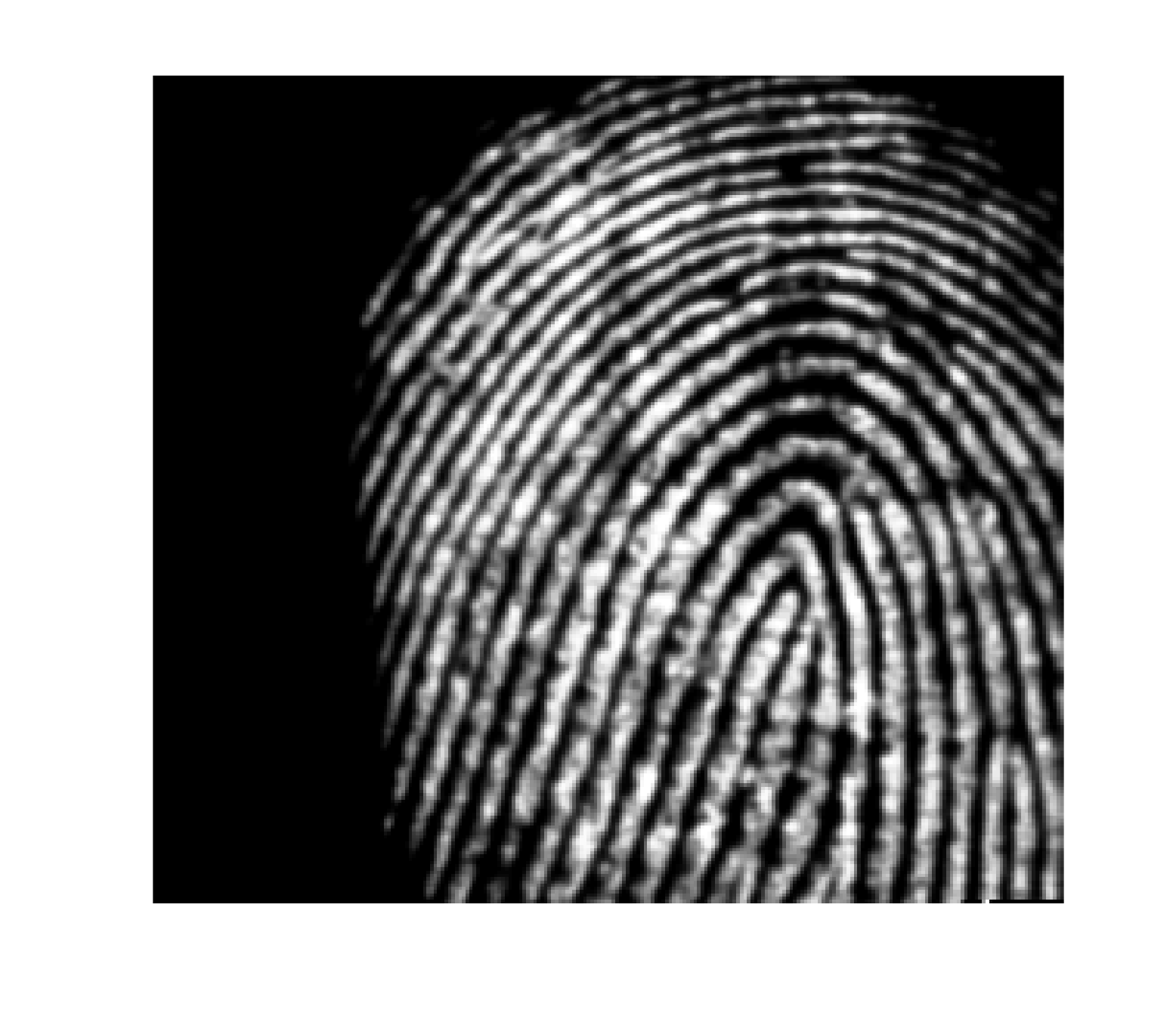}}%
\qquad
\subfloat[][Blurred/noisy image]{\includegraphics[width=6.1cm]{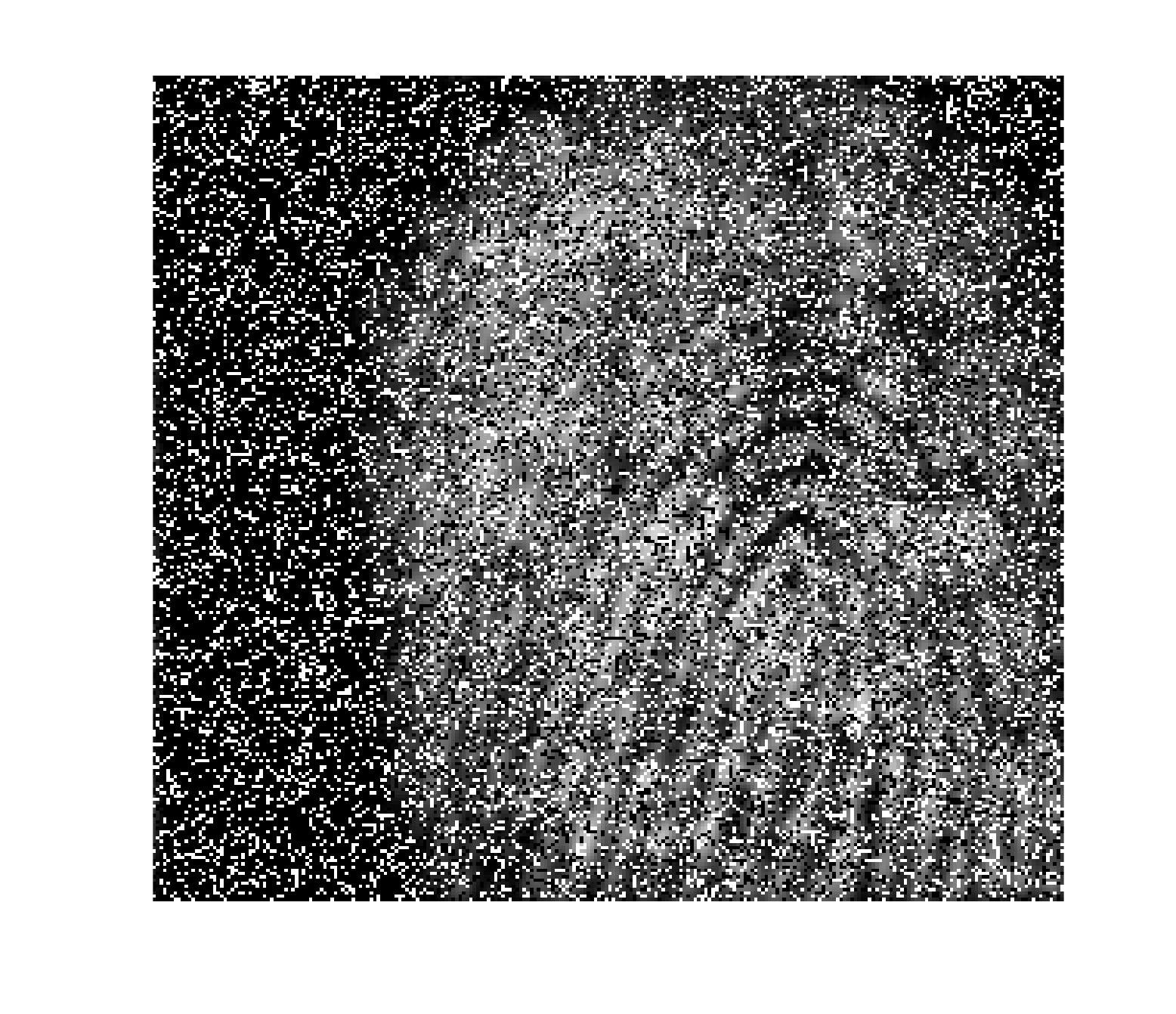}}
\qquad
\subfloat[][DRPD-1: $f = 1.3571e+4, \mathrm{PSNR} = 21.87, \mathrm{T} = 1.28$]{\includegraphics[width=6.1cm]{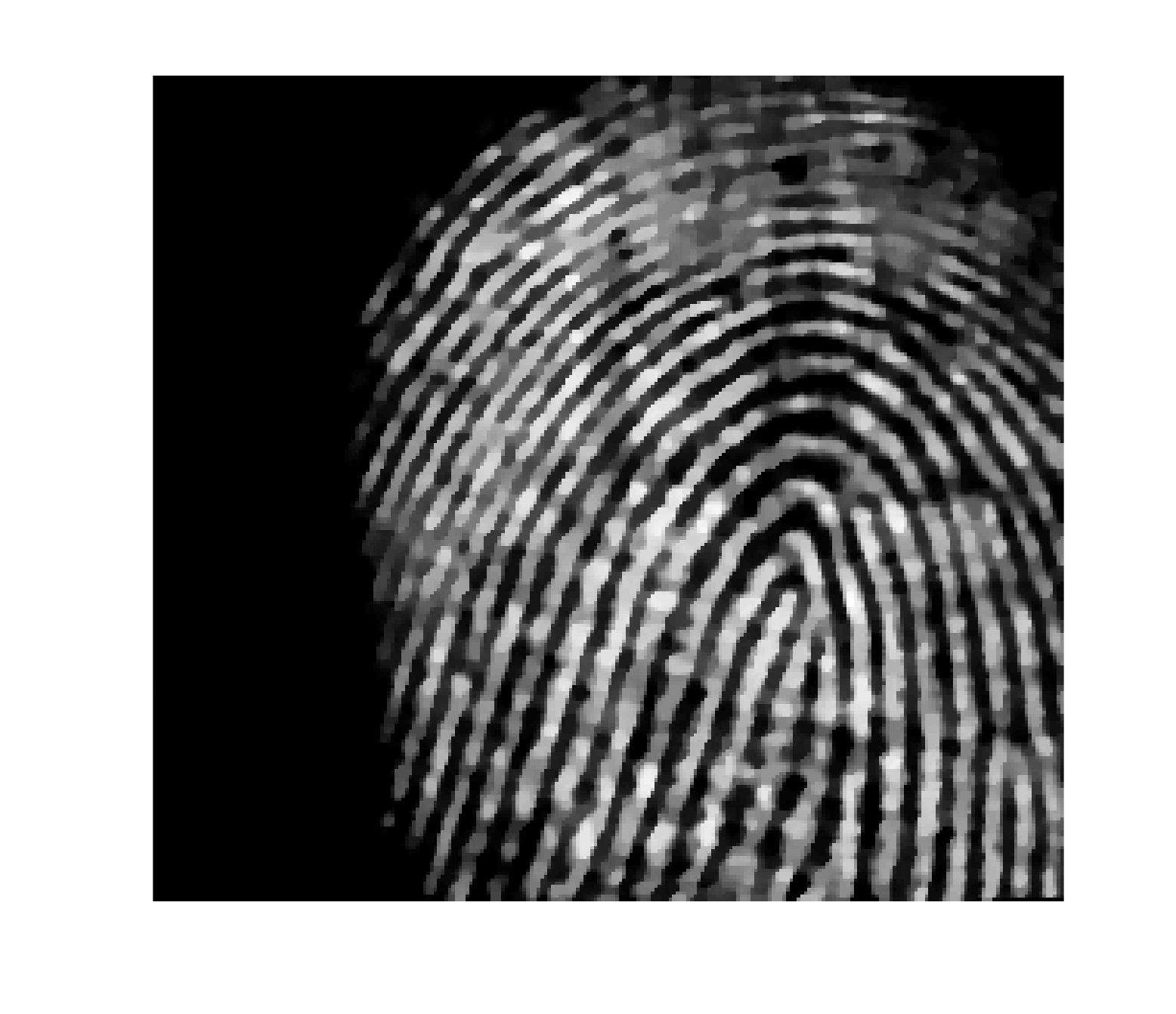}}%
\qquad
\subfloat[][DRPD-2: $f = 1.3635e+4, \mathrm{PSNR} = 21.23, \mathrm{T} = 0.83$]{\includegraphics[width=6.1cm]{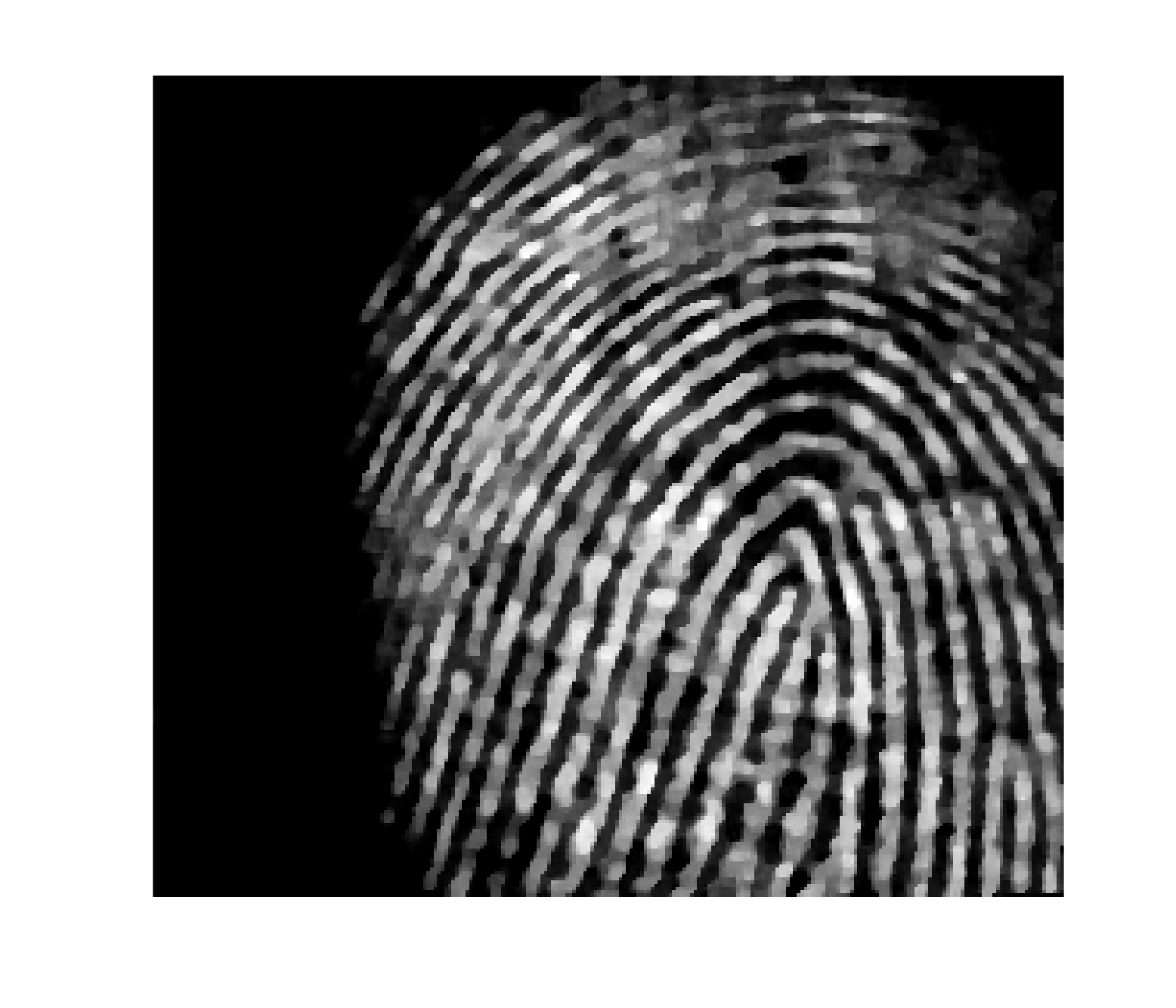}}
\qquad
\subfloat[][ADMM: $f = 1.3799e+4, \mathrm{PSNR} = 20.10, \mathrm{T} = 0.76$]{\includegraphics[width=6.1cm]{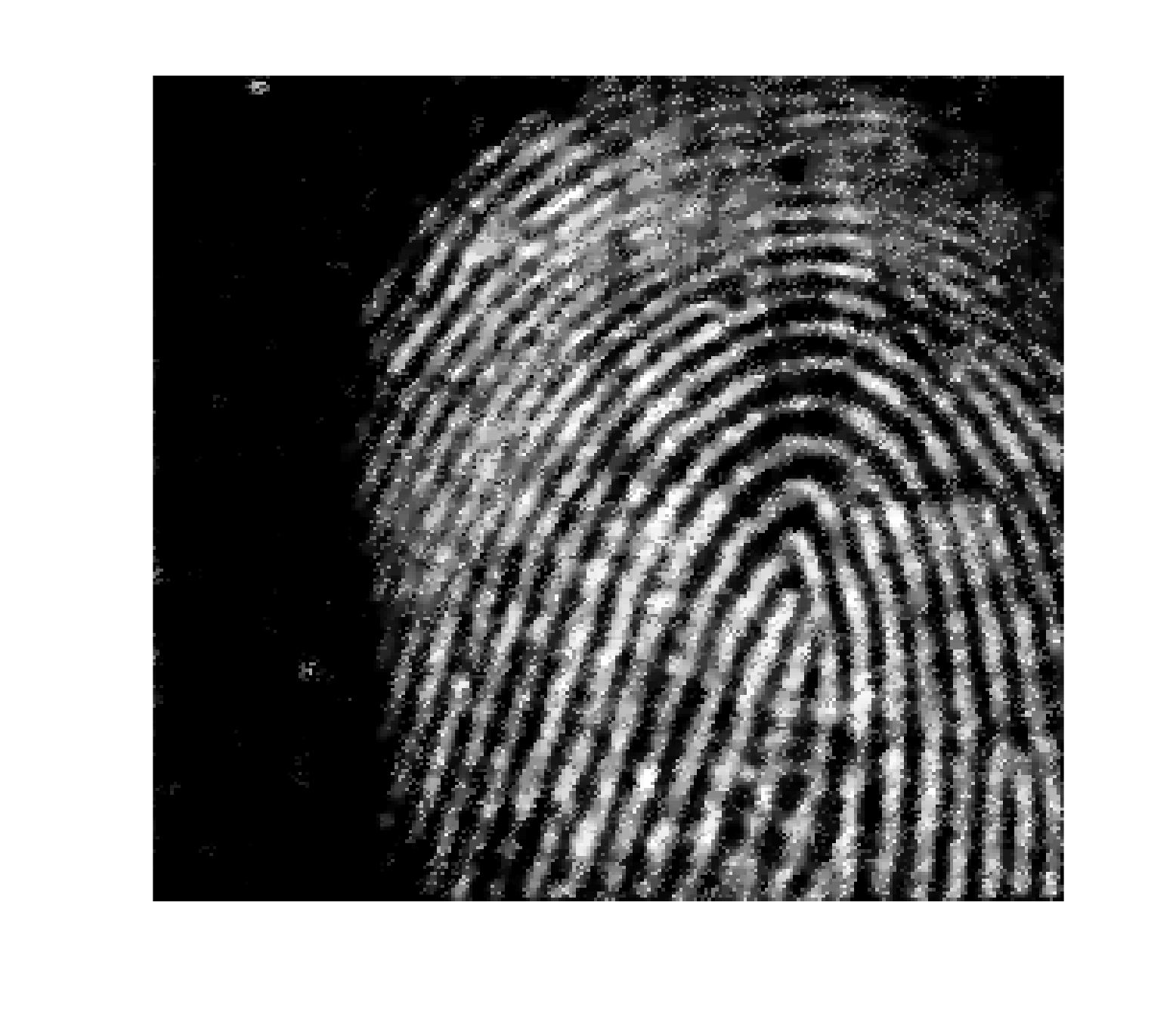}}%
\qquad
\subfloat[][OSGA: $f = 1.3612e+4, \mathrm{PSNR} = 22.04, \mathrm{T} = 6.01$]{\includegraphics[width=6.1cm]{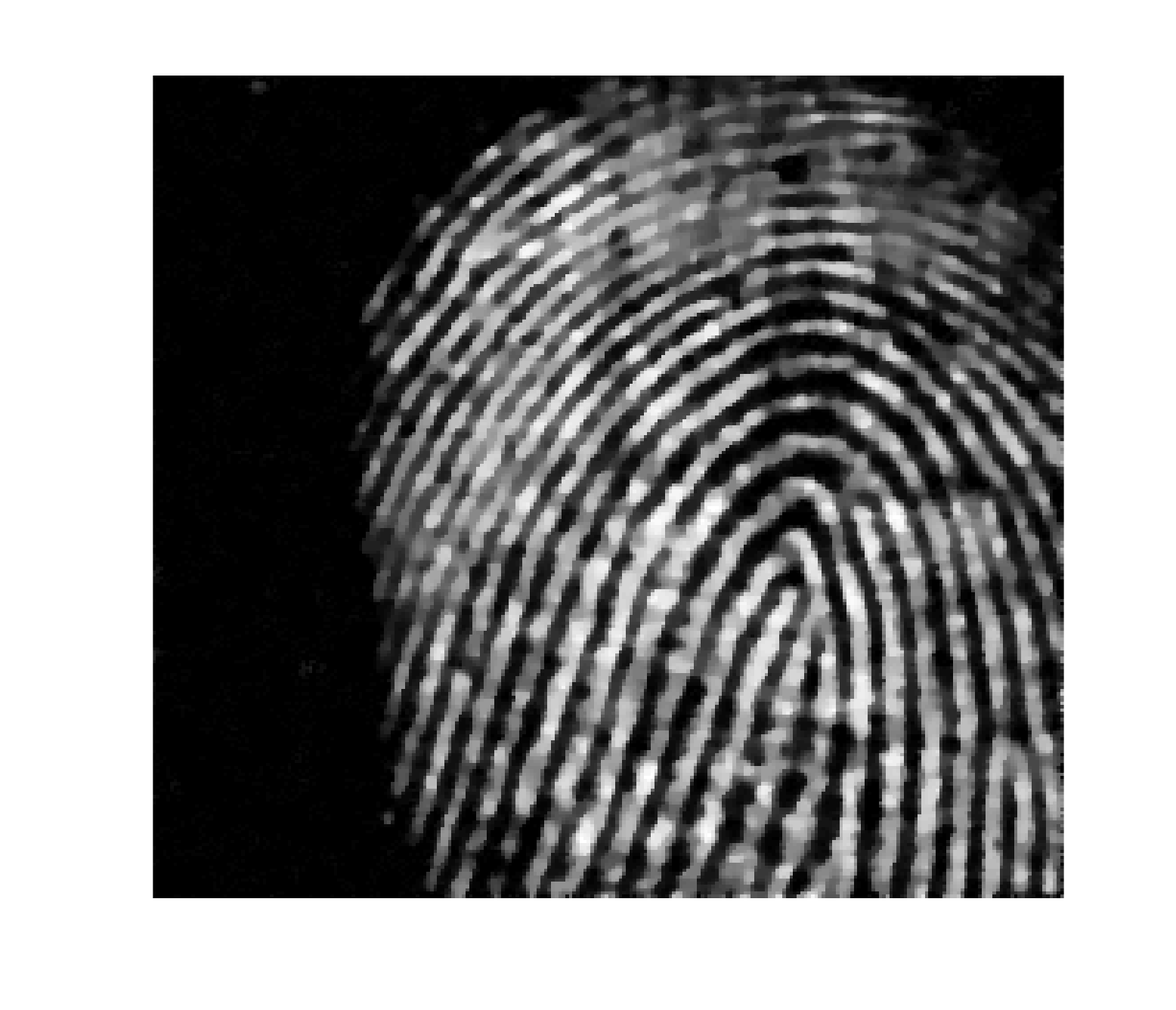}}
\caption{Deblurring of the $256 \times 256$ Fingerprint image using DRPD1, DRPD2, ADMM, and OSGA with the regularization parameter $\lambda = 7 \times 10^{-3}$. The algorithms were stopped after 50 iterations. The blurred/noisy image was constructed by the $7 \times 7$ Gaussian kernel with standard deviation 5 and salt-and-pepper impulsive noise with the level $50 \%$.}%
\end{figure}

\vspace{5mm}
\section{Concluding remarks}
This paper addresses an optimal subgradient algorithm, OSGA,  for solving bound-constrained convex optimization. It is shown that the solution of OSGA's subproblem has a piecewise linear form in a single dimension. Afterwards, we give two iterative schemes to solve this one-dimensional problem, where the first one solves OSGA's subproblem exactly and the second one inexactly. In particular, the first scheme uses the piecewise linear solution to translate the subproblem into a one-dimensional piecewise rational problem. Subsequently, the optimizer of the subproblem is founded in each interval and compared to give the global solution. The second scheme substitutes the piecewise linear solution into the subproblem and give a one-dimensional nonlinear equation that can be solved with  zero finders to give the optimizer. Numerical results are reported showing the efficiency of OSGA compared with some state-of-the-art algorithms. \\\\
{\bf Acknowledgement.} We would like to thank {\sc Radu Bot} and {\sc Min Tao} for making their codes DRPD-1, DRPD-2, and ADMM available for us.


\end{document}